\documentclass[12pt]{article}
\usepackage[dvipsnames]{xcolor}

\usepackage{graphicx,tikz,cancel}
\usepackage[dvipsnames]{xcolor}

\usepackage{amsmath,amssymb,amsthm,amsfonts,amsopn,mathrsfs}
\usepackage{bm}
\usepackage{mathtools}
\usepackage{bigints}
\usepackage{latexsym}
\usepackage{color}

\usepackage{pgfplots}
\pgfplotsset{width=8cm,compat=1.15}

\usepackage{caption,subcaption,listings,multirow,booktabs}
\usepackage[algoruled,lined,linesnumberedhidden,noend]{algorithm2e}
\DontPrintSemicolon

\usepackage{dsfont,microtype,mathrsfs,bigints,stmaryrd,hyperref,url}

\usepackage[english]{babel}

\usepackage[square,numbers,comma,sort&compress]{natbib}

\makeatletter
\newcommand{\citethm}[2][]{%
\begingroup
  \let\NAT@mbox=\mbox
  \let\@cite\NAT@citenum
  \let\NAT@space\NAT@spacechar
  \let\NAT@super@kern\relax
  \renewcommand\NAT@open{[}%
  \renewcommand\NAT@close{]}%
  \cite[#1]{#2}%
  \endgroup
}

\newcommand{\NN}{\mathbb{N}}

\newcommand{\RR}{\mathbb{R}}
\newcommand{\PP}{\mathbb{P}}
\newcommand{\EE}{\mathbb{E}}

\newcommand{\de}{\mathrm{d}}

\newcommand{\abs}[1]{\left\lvert #1 \right\rvert}

\newcommand{\norm}[1]{\left\lVert#1\right\rVert}

\newcommand{\N}[1]{\ensuremath{\{1,\ldots,{#1}\}}}

\newcommand{\eg}{\textit{e.g.}}

\DeclareSymbolFont{calletters}{OMS}{cmsy}{m}{n}
\DeclareSymbolFontAlphabet{\mathcal}{calletters}

\def\be{\begin{eqnarray}}
\def\ee{\end{eqnarray}}

\def\b*{\begin{eqnarray*}}
\def\e*{\end{eqnarray*}}

\newtheorem{Theorem}{Theorem}[section]

\newtheorem{Proposition}[Theorem]{Proposition}

\newtheorem{Assumption}{Assumption}
\newtheorem{Lemma}[Theorem]{Lemma}

\newtheorem{Remark}[Theorem]{Remark}
\newtheorem{Example}[Theorem]{Example}

\makeatletter \@addtoreset{equation}{section}

\usepackage{dsfont}

\newcommand{\bea}{\begin{eqnarray}}
\newcommand{\bes}{\begin{subequations}}
\newcommand{\ees}{\end{subequations}}
\newcommand{\bgt}{\begin{gather}}
\newcommand{\egt}{\begin{gather}}
\newcommand{\eea}{\end{eqnarray}}
\newcommand{\beaa}{\begin{eqnarray*}}
\newcommand{\eeaa}{\end{eqnarray*}}

\def \E{\mathbb{E}}
\def \F{\mathbb{F}}

\def \P{\mathbb{P}}

\def \R{\mathbb{R}}
\def \S{\mathbb{S}}
\def \M{\mathbb{M}}
\def \N{\mathbb{N}}

\def\Ac{{\cal A}}

\def\Cc{{\cal C}}

\def\Fc{{\cal F}}

\def\Lc{{\cal L}}
\def\Mc{{\cal M}}
\def\Nc{{\cal N}}
\def\Oc{{\cal O}}

\def\Wc{{\cal W}}

\def \eps{\varepsilon}

\def \0{\mathbf{0}}

\def \vp{\varphi}
\def \x{\times}

\def \Z{\mathbb{Z}}

\def\1{{\rm 1}}

  \def\vs#1{\vspace{2mm}}

\def\D{{\mathrm{D}}}

\DeclareMathOperator{\Tr}{Tr}

\usepackage[left=0.75in, right=0.75in, top=1in, bottom=1.5in]{geometry}
\usepackage[T1]{fontenc}
\usepackage{hyperref}
\usepackage{textcomp}
   \usepackage{accents}
   \usepackage[mathscr]{euscript}

\def\Ar{{\rm A}}
\def\Ab{{\mathbb A}}
\def\DD{{\mathbb{D}}}
\def\Ac{{\mathcal A}}

\def\ve{\varepsilon}
\def\Trb#1{{\rm Tr}\left[#1\right]}

\def\scap#1#2{\langle #1,#2\rangle}
\def\Wf{{\rm w}}
\usepackage{enumitem}
\title{Diffusive limit approximation of pure jump  optimal ergodic control problems}
\author{Marc Abeille\footnote{Criteo AI Lab. m.abeille@criteo.com}, Bruno Bouchard\footnote{CEREMADE, Universit\'e Paris-Dauphine, PSL, CNRS.  bouchard@ceremade.dauphine.fr. }, Lorenzo Croissant\footnote{CEREMADE, Universit\'e Paris-Dauphine, PSL, CNRS, and Criteo AI Lab.  croissant@ceremade.dauphine.fr.} }
\usepackage[normalem]{ulem} 

\usepackage[section]{placeins} 

\newcommand{\abss}[1]{\lvert #1\rvert}

\begin{document}
\maketitle

\begin{abstract} Motivated by the design of fast reinforcement learning algorithms, we study the diffusive limit of a class of pure jump ergodic stochastic control problems. We show that, whenever the  intensity of jumps is large enough, the approximation error is governed by the H\"older continuity of the Hessian matrix of the solution to the limit ergodic partial differential equation. This extends to this context the results of \cite{abeille2021diffusive} obtained for finite horizon problems. We also explain how to construct a first order error correction term under appropriate smoothness assumptions. Finally, we quantify the error induced by the use of the Markov control policy constructed from the numerical finite difference scheme associated to the limit diffusive problem, this seems to be new in the literature and of its own interest. This approach permits to reduce very significantly the numerical resolution cost.      
\end{abstract}

\section{Introduction}

Let $N$ be a random point process with predictable compensator $\eta \nu(\de e)\de t$, for some finite probability measure $\nu$ on $\RR^{d'}$, $d'\in \N$, $\eta>0$, and let $X^{x,\alpha}$ be the solution of 
\[
X^{x,\alpha}=x+  \int_0^\cdot \int_{\RR^{d'}} b(X^{x,\alpha}_{s-},\alpha_{s},e)N(\de e,\de s)\,, 
\]
 in which $\alpha$ belongs to  the set $\Ac$ of predictable controls with values in some given compact set $\Ab\subset \R^{m}$ and the initial data $x\in \R^d$, $m\in \N$. Under some standard stability assumptions, the value of the ergodic optimal control problem 
 \[ \rho^*:=\sup_{\alpha\in\Ac}\liminf_{T\to+\infty}\frac{1}{\eta T}\EE\left[\int_0^T r(X^{0,\alpha}_{s-},\alpha_s)\de N_s\right]\]
 with $N_t:=N(\RR^{d'},[0,t])$, $t\ge 0$, along with some continuous function $\Wf$, solves the integro-differential equation 
 \begin{align}\label{eq: pde integro intro} 
 \rho^*+ \sup_{a\in \Ab}\left\{ {\eta}\int [\Wf(\cdot+b(\cdot,a,e))-\Wf]\nu(\de e)+r(\cdot,a)\right\}=0 \mbox{ on }\RR^d\,
 \end{align}
 possibly in the viscosity solution sense. This characterisation leads to numerical schemes for approximating the value of the problem and the Markovian optimal control. 
 
 However, \eqref{eq: pde integro intro} is non-local in nature which means that, unless $\nu$ is concentrated on a small number of points,  the cost of numerical approximation is large, in particular when the intensity $\eta$ is. This is a problem, e.g., for bidding problems (see e.g.~\cite{croissant2020real}) in online display-ad auctions, where the system moves near-continuously in time, meaning that $\eta$ is very large, and where unknown system parameters motivate the use of reinforcement learning to solve the control problem. Reinforcement learning compounds the cost by requiring computation of $\rho^*$ for many plausible values of the parameters.

 On the other hand, when $\eta$ is very large, asymptotic regimes exist which offer an alternative approximation path, notably the diffusive limit on which this paper focuses. Indeed, taking $\eta=\ve^{-1}$ and $b(x,a,e)=\ve b_1(x,a,e) + \ve^{\frac12} b_2(x,e)$, with $\int_{\R^{d'}}b_{2}(\cdot,e)\nu(\de e)=0$, an immediate second order expansion shows that $(\rho^*,\Wf)$ converges as $\ve\to0$ to the solution $(\bar\rho^*,\bar \Wf)$ of
 \begin{align}
   \bar\rho^* + \sup_{\bar a\in\Ab}\left\{ \int_{\RR^{d'}}b_1^{\top}(\cdot,\bar a,e)\nu(\de e)\D\bar \Wf+\Tr\left[\int_{\RR^{d'}}b_2b_2^\top(\cdot,e)\nu(\de e)\D^2\bar \Wf\right]+r(\cdot,a)\right\}=0 \mbox{ on }\RR^d.\label{eq: intro HJB diff}
 \end{align}
 Unlike \eqref{eq: pde integro intro}, \eqref{eq: intro HJB diff} is a local equation and much more easily solved numerically. Note that another possible limit regime, albeit less precise, is obtained  via a first order expansion as in \cite{fernandez2017optimal}, which corresponds to considering a fluid limit.
 
  For such a specification of the coefficients ($\eta$, $b$), the existence of a diffusive limit is expected, see e.g.~\cite{jacod2013limit} for general results on the convergence of stochastic processes.  Stability of viscosity solutions, see e.g.~\cite[Section 3]{fleming1989existence}, can also be used to prove the convergence of the value function of stochastic control problems. This has been a subject of particular interest in  insurance and queueing network literatures, see e.g.~\cite{bauerle2004approximation,cohen2020rate,chen_fundamentals_2001}. Nonetheless, these approaches do not permit to characterize  the speed of convergence in the case of a (generic) ergodic optimal control problem as defined in Section \ref{sec: jump} below, which is essential for studying reinforcement learning problems.

The aim of this paper is to characterize this convergence speed and explain how to numerically construct, in an efficient way, an approximation of the optimal control. A first step in this direction was done by  \cite{abeille2021diffusive} who considered  finite time horizon problems. Such problems are  easier to handle from a mathematical point of view, but are unfortunately less adapted to reinforcement learning algorithms.  

Still, a similar approach can be used, up to additional technicalities. As in \cite{abeille2021diffusive}, we study the regularity of $\bar \Wf$ in the solution couple to \eqref{eq: intro HJB diff}. We show that its second order derivative is (locally) $\gamma$-H\"older with a constant of at most linear growth in $x$, for some $\gamma\in(0,1]$, whenever the coefficients of \eqref{eq: intro HJB diff} are uniformly Lipschitz in space,  $\int b_1(\cdot,e)\nu(\de e)$ has linear growth,  $b_2$ and $r$ are continuous and bounded, and under a uniform ellipticity condition. By a second order Taylor expansion, this allows us to pass (rigorously) from \eqref{eq: intro HJB diff}  to \eqref{eq: pde integro intro} up to an error term of order $\ve^{\frac{\gamma}2}$ (locally), and therefore provides the required convergence rate by verification. In general this rate can not be improved. As a by-product, the Markovian control taken from the Hamilton-Jacobi-Bellman equation of the diffusive limit problem provides an $\ve^{\frac{\gamma}{2}}$-optimal control for the original pure-jump control problem. Under additional regularity assumptions, it can even be improved by constructing  a first-order correction term.

In principle, this provides an efficient way of constructing an almost optimal Markovian control. However, it still remains to build up a pure numerical scheme. To complete the picture we therefore derive a convergence rate for a finite difference method for the numerical estimation of $\bar\rho^*$, depending again on $\gamma$. More importantly, we explain how to numerically construct an almost optimal Markovian control process based on a smoothed version of the numerical approximation of $\bar \Wf$ and we obtain a convergence rate towards $\bar\rho^*$, and therefore $\rho^*$,  of the expected average gain associated to such a control. The latter seems to be (surprisingly) completely new and of own interest in the optimal control literature.

As an example of application, we consider in Section \ref{sec:numerical example} a simplified repeated online auction bidding problem, where a buyer seeks to maximise its profit when facing both competition and a seller who adapts its price to incoming bids. Our numerical experiments show that our approximation permits a considerable gain in computation time (as expected).

Note that we restrict here to the case where $b_2$ does not depend on the value of the control, meaning that $\bar \Wf$ solves a semi-linear equation. In principle, the fully non-linear case could be studied along the same lines of arguments but the required regularity of the corresponding function $\bar \Wf$ would be much more complex to derive. We avoid considering this more general case for sake of simplicity (note that standard reinforcement learning problems actually use simple additive noises).


{\bf Notations: } We collect here some standard notations that will be used throughout this paper. Any element $x$ of $\R^{d}$ is viewed as a column vector. $\M^{d}$ (resp.~$\S^{d}$) denotes the collection of (resp. symmetric) $d$-dimensional matrices. On $\R^{d}$ or $\M^{d}$, the superscript $^{\top}$ denotes transposition, we set $\scap{x}{y}:=x^{\top} y$ {and $|x|:=\sqrt{\scap{x}{x}}$} for $x,y\in \R^{d}$. We let $\Tr[M]$ denote the trace of $M\in \M^{d}$ {and $|M|$ be the Euclidean norm of $M$ viewed as a vector of $\R^{d\times d}$}. We  denote by $B_{\ell}(x)$ the open ball centered at $x\in \R^{d}$ of radius $\ell>0$. Given an open set ${\cal \Oc}\subset \R^{n}$, $n\ge 1$, $p\in \{0,1,2\}$, we use the standard notation $\Cc^{p}({\cal \Oc})$ to denote the space of $p$-times continuously differentiable real-valued maps $u$ on ${\cal \Oc}$, and $\Cc^p_b({\cal \Oc})$ to denote the subspace of functions $u\in\Cc^p({\cal \Oc})$ such that 
\[ \norm{u}_{\Cc^p_b({\cal \Oc})}:= \sum_{j=0}^p \sup_{x\in {\cal \Oc}}\abs{\D^ju(x)}<\infty\,\]
in which $\D^{0}u:=u$, $\D^{1}u$ is the gradient of $u$, as a line vector, $\D^{2}u$ is the Hessian matrix of $u$. 
Given $\gamma\in [0,1]$, we denote the $\gamma$-H\"older modulus of $u\in\Cc^0({\cal \Oc})$ on ${\cal \Oc}$ as
\[ [u]^{{\gamma}}_{\Cc^{{0}}({\cal \Oc})}:= \sup_{x,x'\in {\cal \Oc}}\frac{\abs{u(x')-u(x)}}{\abs{x'-x}^\gamma}\,,\]
where we use the convention $0/0=0$. 
{If $u=(u^{1},\cdots,u^{d})$ takes values in $\R^{d}$, $d\ge 1$, we use the same notation to denote the sum of the elements $\{[u^{i}]^{\gamma}_{\Cc^0({\cal \Oc})},i\le d\}$}. We write $u\in \Cc^{p,\gamma}({\cal \Oc})$ if $\D^{p}u$ is $\gamma$-H\"older on each compact subset of ${\cal \Oc}$, and $u\in \Cc^{p,\gamma}_{b}({\cal \Oc})$  if  
$$
\|u\|_{\Cc^{p,\gamma}_{b}({\cal \Oc})}:=\norm{u}_{\Cc^p_b({\cal \Oc})}+[\D^{p} u]^{\gamma}_{{\Cc^{0}(\cal \Oc)}} <\infty.
$$
 If $u$ is restricted to take values in a subset ${\cal \Oc}'$ of $\R$, we write  $\Cc^{p}({\cal \Oc};{\cal \Oc}')$, $\Cc^{p}_{b}({\cal \Oc};{\cal \Oc}')$, $\Cc^{p,\gamma}({\cal \Oc};{\cal \Oc}')$ or $\Cc^{p,\gamma}_{b}({\cal \Oc};{\cal \Oc}')$ for the corresponding sets. We also use the notation $C^{0}_{\rm lin}({\cal \Oc})$  to denote the collection of continuous real-valued function $u$  such that  
$$
[u]_{\Cc^{0}_{\rm lin}({\cal \Oc})}:=\sup_{x \in {\cal \Oc}} \frac{| u(x)|}{1+|x|}  <\infty.
$$
{In all the above notations, we omit $\Oc$ if it is equal to $\R^{d}$.}

\section{Pure jump Ergodic Optimal Control}\label{sec: jump}

In order to alleviate notations, we first consider the case where the intensity of the jump process is given, and recall rather standard results from the ergodic control litterature.  

Let $\Omega=\DD$ denote the space of $d$-dimensional c\`adl\`ag functions on $\RR_+$ and ${\cal M}(\RR^{d'}\times\RR_+)$ denote the collection of positive finite measures on $\RR^{d'}\times\RR_+$, for some $d,d'\in\NN^*$.  Consider a measure-valued  map $N: \DD\mapsto \Mc (\RR^{d'}\times\RR_+)$ and a probability measure $\PP$  on $\DD$ such that  $N$ is a right-continuous real-valued $\RR^{d'}$-marked point process with compensator $\eta\nu(\de e)\de t$, in which   $\eta>0$ and $\nu$ is a probability measure on $\RR^{d'}$. See e.g.~\cite{bremaud1981point}. For ease of notations, we set $N_{t}:=N(\RR^{d'},[0,t])$ for $t\ge 0$.

Let ${\F=(\Fc_{t})_{t\ge 0}}$ be the $\P$-augmentation of the filtration generated by $(\int_{0}^{t} \int_{\R^{d'}}e N(\de e,\de r))_{t\ge 0}$. Given a compact set $\Ab\subset\RR^m$, {$m\in \N$}, let $\Ac$ be the collection of $\F$-predictable processes with values in $\Ab$.  Throughout this paper, unless otherwise stated, we will work on the filtered probability space $(\Omega,\Fc, \F, \P)$, where   $\Fc=\Fc_\infty$. 

Given $(t,x)\in \RR_+\x \RR^d$, $\alpha \in \Ac$, and a measurable map $(x,a,e) \in \RR^d\x \Ab\times \RR^{d'} \mapsto b(x,a,e)\in\RR^{d}$, we  define the c\`adl\`ag process ${X^{x,\alpha}}$ as the solution of 
\begin{align}\label{eq: def X jump}
X^{x,\alpha}_\cdot=x+  \int_0^\cdot \int_{\R^{d'}} b(X^{x,\alpha}_{s-},\alpha_{s},e)N(\de e,\de  s)\,.
\end{align}

We then consider the ergodic gain functional 
\begin{align}
   \rho(x,\alpha):=\liminf_{T\to \infty}\frac1{\eta T}  \EE\left[ \int_0^T   r(X_{t-}^{x,\alpha},\alpha_t)\de N_t \right]\,,\; {(x,\alpha)\in \R^{d}\times \Ac},\label{eq:def ergodic functional jump}
  \end{align}
  for some bounded measurable map $(x,a) \in \RR^d\x \Ab\mapsto r(x,a)\in\RR$. Note that this actually also pertains to the case where the reward function $r$ depends on the mark $e$, by arguing as in Remark \ref{rem:reec V lambda et T} below. By the same remark, the cost could have an extra component given in term of the Lebesgue measure.
  
In the above the scaling by $1/(\eta T)$ means that we consider the gain by average unit of time the controller acts on the system. Indeed, $\E[N_{T}]=\eta T$ and the control applies only at jump times of $N$.
 
This functional induces an infinite horizon control problem corresponding to finding the value function 
\begin{align}\label{eq: eq def rho* ergodic saut} 
\rho^*:=\sup_{\alpha\in\Ac}\rho(\cdot,\alpha).
\end{align}
This problem is meaningfully ergodic when  {$\rho^{*}$ is constant over $\R^{d}$}, {i.e.~the initial condition does not play any role. }

 {All throughout this paper, we make the following assumptions.} First, we impose some control on the coefficients $(b,r)$. 

\begin{Assumption}\label{asmp: basics}  The map $(b,r)$ is  continuous. Moreover, there exists $L_{b,r}>0$    such that 
\begin{equation*}
[b(\cdot,a,e)]_{\Cc^0_{{\rm lin}}}+\norm{r(\cdot,a)}_{\Cc^{0,1}_b} \le L_{b,r},\; \mbox{ for all }  (a,e)\in \Ab\x \R^{d'}.
\end{equation*}
\end{Assumption}


The next assumption, known as assymptotic flatness, guarantees that each control process contracts all possible paths of \eqref{eq: def X jump} exponentially fast to a single trajectory. This is a sufficient condition to ensure that $\rho^{*}$ does not depend on the initial condition. See the proof of Lemma \ref{lemma: Lipschitz value discounted jump} in the Appendix. It can be compared to standard assumptions used in the Brownian diffusion case as in e.g.~\cite[Proof of Lemma 7.3.4]{arapostathis2012ergodic}, up to a more abstract statement.  
\begin{Assumption}\label{asmp: controllability for jumps}
There is $\zeta\in \Cc^{0}(\RR^d\times\RR^d;\RR_+)$ such that 
\begin{enumerate}
\item[(i)] There exists $(\ell_{\zeta},L_{\zeta})\in{(\RR_{+}^*)}^2$ and $p_\zeta \ge 1$ for which 
\[
\ell_{\zeta}|x-x'|^{p_{\zeta}}\le \zeta(x,x') \le L_{\zeta}|x-x'|^{p_{\zeta}},\;\mbox{ for all } x,x'\in \RR^{d}.
\]
\item[(ii)] There exists $C_{\zeta}>0$   such that for all $x,x'\in\RR^d$, $a\in \Ab$ and $\iota>0$ 
  \begin{align}\label{eq: Borkar contraction jump}
  \eta\int_{\R^{d'}}\left\{ \zeta (x+b(x,a,e),x'+b(x',a,e))-\zeta(x,x')\right\}\nu(\de e) \le -C_{\zeta}\zeta(x,x')\,.
  \end{align}
  \end{enumerate}
\end{Assumption}

Our last assumption is typically required to control the long time behavior of solutions of \eqref{eq: def X jump}, see Lemma \ref{lemma: borne EXq} in the Appendix. It is a form of Lyapunov stability assumption, see e.g. \cite{hafstein_lyapunov_2019,borkar_ode_2021} for comparison.
\begin{Assumption}\label{asmp: mean rever}
There is $\xi\in \Cc^{0}(\RR^d\times\RR^d;\RR_+$) such that 
\begin{enumerate}
\item[(i)] There exists $(\ell_{\xi},L_{\xi})\in{(\RR_{+}^*)}^2$ and $p_\xi \ge 1$ for which 
\[
\ell_{\xi}|x|^{p_{\xi}}\le \xi(x) \le L_{\xi}|x|^{p_{\xi}},\;\mbox{ for all } x\in \RR^{d}.
\]
\item[(ii)] There exists $C^{1}_{{\xi}}> 0$ and $C^{2}_{\xi}\in \R$ such that for all $x\in\RR^d$, $a\in \Ab$ and $\iota>0$ 
  \begin{align}\label{eq: Borkar contraction jump pour mean rever}
  \eta\int_{\R^{d'}}\left\{ \xi (x+b(x,a,e))-\xi(x)\right\}\nu(\de e) \le -C^{1}_{\xi}\xi(x)+C^{2}_{\xi}.
  \end{align}
  \end{enumerate}
  \end{Assumption}


\begin{Example}\label{example : mean reverting 1} 
  Consider a bidding problem in a repeated auction with reserve (see \eg~\cite{krishna_auction_2009} for an introduction to auctions), in which $X$ stands for the current reserve price and $\alpha$ is the bid. We set $e=(e_1,e_2,e_3,e_4)\in\RR^4$ and consider the dynamic induced by $b(x,a,e):=e_{1}(ae_2+e_{3}-x)$ for $\Ab:=[\underbar a,\overline a]{\subset \R_{+}}$.  {This means that the dynamic is {\sl mean-reverting} around the level $a e_{2}+e_{3}$. In this formula, $e_{2}$ correspond to the retail value (the price at which the bidder will sell to the final client the product he bought) so that the value $a$ of the control  is the so-called shading factor. Then, $e_{1}\ge 0$ is the realization of a random mean-reversion speed and $e_{3}$ is an exogeneous noise. If the reserve price value $x$ is smaller than the bid price $ae_{2}$ (up to the additional noise $e_{3}$) then it moves up for the next auction, and the other way round if it is bigger.} In a second price auction, with $e_4$ as the value of the competition bid, the natural reward function is
  \[ 
  r(x,a)=\int_{\R^{4}} (e_2-x\vee e_4)\1_{\{a e_2\ge x\vee e_4 \}} \nu(\de e)\,.
  \]
  We assume that $\nu([0,1] {\x \R_{+}}\x \RR^2)=1$,  ${1-\int_{\R^{4}} (1-e_{1})^{2p} \nu(de)}=:m_{1}\in (0,1]$ and that $\int_{\R^{4}} \sup_{a\in \Ab} |a  {e_{1}}e_2+{e_{1}}e_{3}|^{2p}\nu (de){<}\infty$, for some integer $p\ge 1$. Then, Assumption \ref{asmp: controllability for jumps} holds with $\zeta(x,x'):=|x-x'|^{2p}$ and $C_{\zeta}=\eta m_{1}$, while Assumption \ref{asmp: mean rever} holds with $\xi(x)=|x|^{2p}$, $C^{1}_{\xi}={\frac12}\eta m_{1}$ and $C^{2}_{\xi}=\eta C_{e}$ {for some $C_{e}>0$ that does not depend on $\eta$}.

Under a standard log-normal model for valuations (see \eg~\cite{ostrovsky_reserve_2011}), and a uniform competition on $[0,\bar c]$ for some $\bar c>0$, it is easily verified that Assumption \ref{asmp: basics} holds. This example is developped further in Section \ref{sec:numerical example}.
\end{Example}

Under the above assumptions, we obtain the following classical result, Theorem \ref{thm: jump ergodic rho HJB + verif} below whose proof is rather standard, but produced in   the Appendix by lack of an appropriate reference.  To state it, we first need to introduce the following auxiliary optimal control problems, defined for all $x\in \R^{d}$, $\lambda,T>0$ and $t\le T$:
\begin{align}\label{eq: def Vlambda} 
V_{\lambda}(x):=\sup_{\alpha\in \Ac}  J_{\lambda}(x,\alpha) \;\mbox{ with } \;J_{\lambda}(x,\alpha):=\frac1\eta \E\left[\int_{0}^{\infty}  e^{-\lambda s}r(X_{s-}^{x,\alpha},\alpha_{s})\de N_{s}\right]  
\end{align}
and 
\begin{align}\label{eq: def VT} 
V_{T}(t,x):=\sup_{\alpha\in \Ac}  J_{T}(t,x,\alpha) \;\mbox{ with } \;J_{T}(t,x,\alpha):=\frac1\eta\E\left[\int_{t}^{T}  r(X_{s-}^{t,x,\alpha},\alpha_{s})\de N_{s}\right]  \,.
\end{align}

\begin{Remark}\label{rem:reec V lambda et T} {Note that Assumption \ref{asmp: basics} implies that $\sup_{[0,t]}|X^{x,\alpha}|$ has moments of any order, for all $t\ge 0$, $(x,\alpha)\in \R^{d}\x \Ac$. Also, it follows from the Assumption \ref{asmp: basics} again and the fact that $\nu$ is a probability measure that 
\begin{align*}
& \rho(x,\alpha)=\liminf_{T\to \infty}  \frac1{ T}\EE\left[ \int_0^T   r(X_{s}^{x,\alpha},\alpha_s)ds\right],\\
&J_{\lambda}(x,\alpha)= \E\left[\int_{0}^{\infty}  e^{-\lambda s}r(X_{s}^{x,\alpha},\alpha_{s})ds\right]  
,\;\mbox{ and }\;
J_{T}(t,x,\alpha)= \E\left[\int_{t}^{T}  r(X_{s}^{t,x,\alpha},\alpha_{s}) ds\right]  \,.
\end{align*}}
For the same reason,  we could consider expected gains of the more general form 
$$
\frac1{\eta T}\EE\left[ \int_0^T  \int_{\R^{d'}}\tilde r(X_{s-}^{x,\alpha},\alpha_s,e) N(\de e,\de s) \right]= \frac1{ T}\EE\left[ \int_0^T  \int_{\R^{d'}}\tilde r(X_{s}^{x,\alpha},\alpha_s,e) \nu(\de e)\de s \right]
$$
upon replacing $r$ by $(x,a)\in \R^{d}\x \Ab \mapsto \int_{\R^{d'}}\tilde r(x,a,e) \nu(\de e)$.  
\end{Remark}

\begin{Theorem}\label{thm: jump ergodic rho HJB + verif}
  Let Assumptions \ref{asmp: basics}, \ref{asmp: controllability for jumps} and \ref{asmp: mean rever} hold. Then, there exists sequences $(\lambda_n)_{n\ge1}$ going to $0$ and $(T_{n})_{n\ge 1}$ going to $+\infty$ such that $(\lambda_n V_{\lambda_n})_{n\ge 1}$ and $(T_{n}^{-1}V_{T_n}(0,\cdot))_{n\ge 1}$ converge  uniformly on compact sets to   $\rho^{*}(0)$, and    such that   $(V_{\lambda_n} - V_{\lambda_n}(0))_{n\ge1}$ converges uniformly on compact sets  to a function $\Wf\in \Cc^{0,1}$ that  solves
  \begin{align}\label{eq: HJB ergodic jump thm}
\rho^{*}&=\sup_{a\in\Ab}\left\{  \eta\int_{\R^{d'}}\left[  \Wf(\cdot+b(\cdot,a,e))-  \Wf\right]\nu(de) +r(\cdot,a)\right\},\;\mbox{ on } \R^{d}.
  \end{align}
  Moreover, 
 $\rho^{*}$ is constant over $\R^{d}$, and, if $(\tilde \Wf,\tilde \rho )\in C^{0}_{\rm lin}\x \R$ solves the ergodic equation 
  \begin{align}\label{eq: HJB ergodic jump}
  \tilde \rho &=\sup_{a\in\Ab}\left\{\eta\int_{\R^{d'}}[\tilde \Wf(\cdot+b(\cdot,a,e))-\tilde \Wf] \nu(\de e)+r(\cdot,a)\right\} ,\;\mbox{ on } \R^{d},
  \end{align}
  then $\tilde \rho=\rho^{*}$.
\end{Theorem}

 \begin{Remark}\label{rem : existence limit dans rho et controle optimal} As a by-product  of Theorem \ref{thm: jump ergodic rho HJB + verif} and   the first part of the proof of Lemma \ref{lem: jump ergodic convergence and verification}, for all $x\in \R^{d}$, there exists an optimal Markovian control defined by  $\hat \alpha:=\hat {\rm a}(X^{x, \hat \alpha}_{{\cdot -}})  $ in which $\hat {\rm a}$ is a measurable map satisfying
\begin{align*} 
 \eta\int_{\R^{d'}}   \Wf(\cdot+b(\cdot,\hat {\rm a}(\cdot),e))   \nu(\de e) +r(\cdot,\hat {\rm a}(\cdot)) =  \max_{a\in \Ab}\left\{ \eta\int_{\R^{d'}}   \Wf(\cdot+b(\cdot,{\rm a},e))   \nu(\de e) +r(\cdot,{\rm a})\right\} ,\;\mbox{ on } \R^{d}.
\end{align*} 
Moreover, 
$$
\rho^{*}=\lim_{T\to \infty }\frac{1}{\eta T }\EE\left[\int_0^T  r(X^{x,\hat \alpha}_{t-},\hat \alpha_{t})\de N_{t} \right].
$$
 \end{Remark}

\section{Approximation for models with large activity }\label{sec: diffusion}

Given an $\varepsilon\in (0,1)$, we   now   replace $\eta$ by  
\[\eta_{\varepsilon}:=\varepsilon^{-1}\,.\]

In the following, we omit the dependence of $N$ and $X^{x,\alpha}$ on $\varepsilon$ for ease of notations and set 
\[ 
\rho_\ve^{*}:=\sup_{\alpha\in\Ac}\liminf_{T\to \infty} \frac{1}{\eta_\ve T }\EE\left[\int_0^T  r(X^{0,\alpha}_{t-},\alpha_{t})\de N_{t} \right]\,.
\]
 We shall see that $\rho^{*}_\varepsilon$, together with the associated optimal policy, can be approximated by considering its diffusive limit as $\varepsilon\to 0$, upon assuming that the jump coefficient $b:=b_\ve$ introduced in Section \ref{sec: jump} is of the form
\[b_{\varepsilon}=\varepsilon b_{1}+\sqrt{\varepsilon} b_{2}\,,\]
and making the following assumption. 

\begin{Assumption}\label{asmp: diffusion limit existence}  
We have $b =\ve b_1  + \sqrt{\ve} b_2 $ for some continuous functions $b_1:\R^{d}\x \Ab\x\R^{d'}\mapsto \R^{d}$ and $b_2:\R^{d}\x\R^{d'}\mapsto \R^{d}$ such that:
\begin{enumerate}
\item[{\rm (i)}] There exists $L_{b_{1},b_{2}}>0$ such that  
$$[b_1(\cdot,a,e)]_{\Cc^0_{{\rm lin}}} + \norm{b_2(\cdot,e)}_{\Cc^0_b}\le L_{b_{1},b_{2}}$$ for all $(a,e)\in\Ab\x\RR^{d'}$. 
\item[{\rm (ii)}] There exists $\varsigma>0$ such that 
\[ \int_{\R^{d'}} b_2(\cdot,e)\nu(\de e)=0 \mbox{ and }  \int_{\R^{d'}} b_2(\cdot,e)b_2(\cdot,e)^\top \nu(\de e) \ge \varsigma\bm{I}_d\, \]
where  $\bm{I}_d$ is the identity matrix.
\item[{\rm (iii)}] The map
\begin{align*}
 (x,a)\in \RR^d\times\Ab  \mapsto \mu(x,a):=\int_{\R^{d'}} b_1(x,a,e)\nu(\de e)
\end{align*}
is Lipschitz in $x$ uniformly in $a$, {and there exists  a Lipschitz {$\R^{d\x d}$}-valued function $\sigma$ defined on $\R^{d}$ such that }
\begin{align*}
\sigma\sigma^{\top}=\int_{\R^{d'}} b_2(\cdot,e)b_2^\top (\cdot,e)\nu(\de  e).
\end{align*}
\item[{\rm (iv)}] The estimates of Assumptions \ref{asmp: basics}, \ref{asmp: controllability for jumps} and \ref{asmp: mean rever} hold for each $(\eta_{\eps,}b_{\eps},r)$ in place of $(\eta,b,r)$, uniformly in $\eps>0$.  
\end{enumerate}
\end{Assumption}

\begin{Example}\label{example : mean reverting 3}Consider the context of Example \ref{example : mean reverting 1} in which  $\eta=\eps^{-1}$ and 
$$
b_{\eps}(x,a,e)=e_{1}(\eps (e_{2}a-x)+\eps^{\frac12} e_{3}), \;(x,a,e)\in \R^{d}\x \Ab\x \R^{4}
$$ 
with $\nu$ as in Example \ref{example : mean reverting 1} such that in addition $\int_{\R^{4}} e_{1}e_{3}\nu(\de e)=0$. 
In this context, we obtain $\mu(x,a)=n_{2}a-n_{1}x$, with $n_{1}:=\int_{\R^{4}}e_{1}\nu(de)$ and $n_{2}:= \int_{\R^{4}} e_{1}e_{2} \nu(\de e)$, and $\sigma(x)^{2}= \int_{\R^{4}} |e_{1}e_{3}|^{2}\nu(\de e)$.
\\
Assume that $n_{1}>0$. Using a second order Taylor expansion around $\eps=0$, one easily checks that Assumption \ref{asmp: mean rever} holds with $\xi(x)=|x|^{2p}$, $p\ge 1$, for some $C^{1}_{\xi}$ and $C^{2}_{\xi}$ that do not depend on $\eps>0$. Similarly, Assumption \ref{asmp: controllability for jumps} holds with $\zeta(x,x')=|x-x'|^{2p}$, $p\ge 1$, for some $C_{\zeta}>0$, uniformly in $\eps\in (0,\eps_{\circ})$, for some $\eps_{\circ}>0$ small enough.
\end{Example}

\subsection{Candidate diffusion limit}\label{subsec: candidate diff limit}

 Let $\bar \P$ be a probability measure on $\DD$ and let $W$ be a stochastic process such that $W$ is a $\bar \P$-Brownian motion,  let {$\bar \F=(\bar \Fc_{s})_{s\ge 0}$} be the $\bar \P$-augmentation of the filtration generated by {$W$}, and let $\bar \Ac$ be the collection of $\bar \F$-predictable processes. Given $\bar \alpha \in \bar \Ac$, we can then define $\bar X^{x,\bar\alpha}$ as the unique strong solution (see \cite[Thm. 1]{veretennikov1981strong}) of 
\begin{align}\label{eq: def X continu}
\bar X^{x,\bar\alpha}=x+\int_{0}^{\cdot} \mu(\bar X^{x,\bar\alpha}_{s},\bar\alpha_{s})\de s +\int_{0}^{\cdot} \sigma(\bar X^{x,\bar\alpha}_{s})\de W_{s}\,.
\end{align}%

The corresponding ergodic control problem is defined by 
$$
\bar \rho^{*}(x):=\sup_{\bar \alpha\in\bar\Ac}\liminf_{T\to \infty} \frac{1}{T  }\EE\left[\int_0^T  r(\bar X^{{x},\bar \alpha}_{t},\bar \alpha_{t})\de {t} \right]\,,\;x\in \R^{d}.
$$

As in Section \ref{sec: jump},  we define for $\lambda>0$ and $x \in  \R^{d}$ 
\[ 
\bar V_{\lambda}(x):=\sup_{\bar \alpha\in \bar \Ac}  \bar J_{\lambda}(x,\bar \alpha) \;\mbox{ with } \;\bar J_{\lambda}(x,\bar \alpha):=\E\left[\int_{0}^{\infty}  e^{-\lambda s}r(\bar X_{s}^{x,\bar \alpha},\bar \alpha_{s})\de s\right],
\]
and impose conditions corresponding to the estimates of Lemma  \ref{lemma: Lipschitz value discounted jump} and \ref{lemma: borne EXq}.  

\begin{Assumption}\label{assumption: Lipschitz value discounted jump + borne EXq} There exists $L_{\bar V},C_{\bar X}>0$ and $p_{\bar X}\ge 1$ such that:
\\
{\rm (i)} For all $x,x'\in \R^{d}$ and   $\lambda \in (0,1)$, 
$$
|\bar V_{\lambda}(x)-\bar V_{\lambda}(x')|\le L_{\bar V}|x-x'|.
$$
{\rm (ii)} For all $x\in \R^{d}$ and   $\bar \alpha \in \bar \Ac$, 
\begin{align*} 
 \E[|\bar X^{x,\bar \alpha}_{t}|^{p_{\bar X}}]&\le C_{\bar X}\left\{e^{- t / C_{\bar X} }|x|^{p_{\bar X}}+ 1\right\} ,\;t\ge 0.
\end{align*}
 \end{Assumption}
 
\begin{Remark}\label{rem: pour hyp lipschiz et controle EXq diffusif} 
{\rm (i)} The condition {\rm (i)} of Assumption \ref{assumption: Lipschitz value discounted jump + borne EXq}  holds for instance under \cite[Assumption 7.3.1]{arapostathis2012ergodic}. Indeed, the latter implies a similar bound as \eqref{eq: borne contractance X-X'}, see \cite[Lemma 7.3.4]{arapostathis2012ergodic}, and the estimate of {\rm (i)} then follows from the same arguments as in the proof of Lemma \ref{lemma: Lipschitz value discounted jump}. More generally, it suffices to find a family of $\Cc^{2}(\R^{d}\times \R^{d};\R)$-functions $(\bar \zeta_{\iota})_{\iota>0}$ that is locally bounded, satisfies 
\begin{align}\label{eq: diffusion contraction eqn}
   \D \bar \zeta_{\iota}(x,x')   \begin{pmatrix} \mu(x,a) \\ \mu(x',{a}) \end{pmatrix}   +\frac12\Tr\left[  {\Sigma}(x,x') \D^2\bar \zeta_{\iota}(x,x')\right]\le -C_{{\bar \zeta}} \bar \zeta_{\iota}(x,x')+\varrho_{\iota} \,,\;   x,x'\in \R^{d},\; a\in \Ab,\;\iota>0, 
\end{align}
in which $C_{{\bar \zeta}}>0$, $\lim_{\iota\to 0} \varrho_{\iota}=0$ and 
$$
 \Sigma(x,x'):=\left(\begin{array}{c}\sigma(x)\\\sigma(x')\end{array}\right)\left(\begin{array}{c}\sigma(x)\\\sigma(x')\end{array}\right)^{\top},
$$
and such that $(\bar \zeta_{\iota})_{\iota>0}$ converges pointwise as $\iota \to 0$ to a map $\bar \zeta:\R^{d}\x \R^{d}\mapsto \R$ satisfing 
$$
\frac1{C_{\bar \zeta}}|x-x'|^{p_{\bar \zeta}}\le \bar \zeta(x,x')\le C_{\bar \zeta}|x-x'|^{p_{\bar \zeta}}\,,\; \mbox{for all } x,x'\in \R^{d},
$$
for some  $p_{\bar \zeta}\ge 1$. This follows from the arguments used in the proof of Lemma \ref{lemma: Lipschitz value discounted jump} upon first applying It\^{o}'s lemma to $\bar \zeta_{\iota}$ and then sending $\iota\to 0$ to deduce the counterpart of \eqref{eq: pre borne contractance X-X'} before using the  inequalities just above.
\\
{\rm (ii)} The condition {\rm (ii)} of Assumption \ref{assumption: Lipschitz value discounted jump + borne EXq}    holds for instance if we can find a smooth function $\bar \xi$ and constants $C^{1}_{\bar \xi}>0$ and $C^{2}_{\bar \xi}$ such that 
 \begin{align}\label{eq:cond suff mean revert bar V}
  \D\bar \xi (x)\mu(x,a)+\frac12\Tr\left[   {\sigma\sigma^{\top}}(x) \D^2\bar \xi(x)\right] \le -C^{1}_{\bar \xi}\bar \xi(x)+C^{2}_{\bar \xi},
  \end{align}
  and 
  \begin{align}\label{eq:cond suff mean revert bar V 2}
 \frac{1}{C^{2}_{\bar \xi}}|x|^{p_{\bar \xi}}\le \bar \xi(x) \le C^{2}_{\bar \xi}|x|^{p_{\bar \xi}},
\end{align}
   for all $x\in \RR^{d}$, for some $p_{\bar \xi}\ge 1$. This follows from the same arguments as in the proof of Lemma \ref{lemma: borne EXq}. As in {\rm (i)} above, it suffices that \eqref{eq:cond suff mean revert bar V} holds for a sequence of approximating smooth functions. In particular, condition {\rm (ii)} of Assumption \ref{assumption: Lipschitz value discounted jump + borne EXq} holds under  \cite[Assumption 7.3.1]{arapostathis2012ergodic}, see  \cite[Lemma 7.6.3]{arapostathis2012ergodic}.
\end{Remark}

\begin{Example}\label{example : mean reverting 2}  Consider the context of Example \ref{example : mean reverting 3} with $\sigma$  constant, then it satisfies  \cite[Assumption 7.3.1]{arapostathis2012ergodic}, and therefore Assumption \ref{assumption: Lipschitz value discounted jump + borne EXq}, by \cite[Example 7.3.3]{arapostathis2012ergodic}. 
\end{Example}

In order to state the   counterpart of Theorem \ref{thm: jump ergodic rho HJB + verif} for the diffusive limit ergodic control problem, we also define,  for $T>0$, $t\le T$ and  $x\in   \R^{d}$, 
\[ 
\bar V_{T}(t,x):=\sup_{\bar \alpha\in \bar \Ac}  \bar J_{T}(t,x,\bar \alpha) \;\mbox{ with } \;\bar J_{T}(t,x,\bar \alpha):=\E\left[\int_{t}^{T}  r(\bar X_{s}^{t,x,\bar \alpha},\bar \alpha_{s})\de s\right]  \,,
\]
and set 
$$
\bar \Lc^{\bar a}\vp=\D \vp \mu(\cdot,\bar a)+\frac12 \Tr[ {\sigma\sigma^{\top}} \D^{2}\vp], \;\bar a\in \Ab,
$$
for a smooth function $\vp:\R^{d}\to \R$.

\begin{Theorem}\label{thm: viscosity solution of ergodic diffusion PDE}
  Let Assumptions \ref{asmp: diffusion limit existence} and \ref{assumption: Lipschitz value discounted jump + borne EXq} hold. Then,
there exists sequences $(\lambda_n)_{n\ge1}$ going to $0$ and $(T_{n})_{n\ge 1}$ going to $+\infty$ such that $(\lambda_n \bar V_{\lambda_n})_{n\ge 1}$ and $(T_{n}^{-1}\bar V_{T_n}(0,\cdot))_{n\ge 1}$ converge  uniformly on compact sets to   $\bar \rho^{*}(0)$, and such that      $( \bar V_{\lambda_n} -  \bar V_{\lambda_n}(0))_{n\ge1}$ converges uniformly on compact sets  to a  function $ \bar \Wf\in \Cc^{2} \cap \Cc^{0}_{{\rm lin}} $   that  satisfies
  \begin{align}\label{eq: HJB ergodic diffusive thm}
 \bar \rho^{*}&=\sup_{\bar a\in\Ab}\left\{\bar \Lc^{\bar a}\bar \Wf+r(\cdot,\bar a)\right\} \,,\mbox{ on } \R^{d},
  \end{align}
  and
  \begin{align}\label{eq: estimee bar W}
\|\bar \Wf\|_{\Cc^{0,1}_{b}}  \le L^{\gamma}_{\bar \Wf}\mbox{ and }  \|\bar \Wf\|_{\Cc^{2,\gamma}_{b}(B_{1}(x))}  \le L^{\gamma}_{\bar \Wf}(1+|x|), \;\mbox{ for all $x\in \R^{d}$,} 
  \end{align}
  for some $L^{\gamma}_{\bar \Wf}>0$, for all $\gamma \in (0,1)$. 
 Moreover, 
 $\bar \rho^{*}$ is constant over $\R^{d}$, and, if $(\tilde \Wf,\tilde \rho )\in (\Cc^{2}  \cap \Cc^{0}_{{\rm lin}})  \x \R$ solves the ergodic equation 
  \begin{align}\label{eq: HJB ergodic jump}
  \tilde \rho &=\sup_{\bar a\in\Ab}\left\{\bar \Lc^{\bar a} \tilde \Wf+r(\cdot,\bar a)\right\}  ,\;\mbox{ on } \R^{d},
  \end{align}
  then $\tilde \rho=\bar \rho^{*}$.
\end{Theorem}
  
\begin{proof} The proof is exactly the same as the one of Theorem \ref{thm: jump ergodic rho HJB + verif} upon replacing the estimates of Lemmas \ref{lemma: Lipschitz value discounted jump} and \ref{lemma: borne EXq} by the ones of Assumption \ref{assumption: Lipschitz value discounted jump + borne EXq}. See the Appendix. The only significant difference is that we have to show  the estimate \eqref{eq: estimee bar W}. 

{1.} The fact that, for an appropriate sequence $(\lambda_{n})_{n\ge 0}$ that converges to $0$, $\lambda_{n} \bar V_{\lambda_{n}}(0)\to c\in \R$ and $ \bar V_{\lambda_{n}}-\bar V_{\lambda_{n}}(0)\to \bar \Wf$ uniformly on compact sets for some $\bar \Wf\in \Cc^{0,1}$   follows from Assumption \ref{assumption: Lipschitz value discounted jump + borne EXq}   and the same arguments as in the first part of the proof of Lemma \ref{lem: jump ergodic rho HJB W} below. 

{2.} We now argue as in the proof of \cite[Theorem 3.5.6]{arapostathis2012ergodic}. {Fix $n\ge 1$, let $\bar \tau^{x,\bar \alpha}_{n}$ be the first exit time of $\bar X^{x,\bar\alpha}$ from $B_{n}(0)$, for $(x,\bar \alpha)\in \R^{d}\x\bar \Ac$, and set 
\[ 
\bar V^{n}_{\lambda}(x):=\sup_{\bar \alpha\in \bar \Ac}  \E\left[\int_{0}^{\bar\tau_{n}^{x,\bar \alpha}}  e^{-\lambda s}r(\bar X_{s}^{x,\bar \alpha},\bar \alpha_{s})\de s\right].
\]
Then, $\bar V^{n}_{\lambda} \in \Cc^{2}(B_{n}(0))$ by the arguments in the proof of \cite[Theorem 3.5.6]{arapostathis2012ergodic}. Moreover, Assumption \ref{assumption: Lipschitz value discounted jump + borne EXq} and the linear growth of $r$ (recall that it is assumed Lipschitz) imply that
$$
\sup_{n\ge 1} [\bar V^{n}_{\lambda}]_{\Cc^0_{{\rm lin}}}\le C_{\lambda}
$$ 
for some $C_{\lambda}>0$. Then, arguing as in the proof of \cite[Theorem 3.5.6]{arapostathis2012ergodic}, we obtain that, for all $\lambda>0$, $(\bar V^{n}_{\lambda} )_{n\ge 1}$ converges as $n\to \infty$ to a map $\psi_{\lambda}\in \Cc^{2}$ that solves  
  \begin{align*} 
    \lambda \psi_{\lambda} &=\sup_{\bar a\in\Ab}\left\{\bar \Lc^{\bar a}  \psi_{\lambda} +r(\cdot,\bar a)\right\}  ,\;\mbox{ on } \R^{d},
  \end{align*}
  and has at most linear growth. Using this linear growth property, Assumption \ref{assumption: Lipschitz value discounted jump + borne EXq}  and a verification argument, we deduce that $\psi_{\lambda}=\bar V_{\lambda}$.}
  
{Since $\bar V_\lambda\in\Cc^{0,1}_b$ by Assumption \ref{assumption: Lipschitz value discounted jump + borne EXq}, it follows from Assumption \ref{asmp: diffusion limit existence} and} Lemma \ref{lem: estimate linear PDE 2} that, given $\gamma\in (0,1)$, $\bar V_{\lambda}\in {\Cc^{2,\gamma}}$ and that there is $K >0$ (depending on $\gamma$ but not on $\lambda\in (0,1)$) such that 
\begin{align}\label{eq: estime C2+1 bar V lambda} 
 \| {\Delta} \bar V_{\lambda}\|_{\Cc^{2,{\gamma}}_{b}(B_{1}(x))}  \le K(1+\abs{x}), \;\mbox{ for all $(x,\lambda)\in \R^{d}\x (0,1)$,} 
\end{align}
 {where}   $\Delta \bar V_{\lambda}:=\bar V_{\lambda}-\bar V_{\lambda}(0)$ solves 
  \begin{align*} 
    \lambda \bar V_{\lambda}(0)+\lambda  \Delta \bar V_{\lambda} &=\sup_{\bar a\in\Ab}\left\{\bar \Lc^{\bar a}  \Delta \bar V_{\lambda} +r(\cdot,\bar a)\right\}  ,\;\mbox{ on } \R^{d}.
  \end{align*}
  Let   $(\lambda_{n})_{n\ge 0}$ be as in step 1. 
Passing to the limit in the above leads to \eqref{eq: HJB ergodic diffusive thm}, with $c$ defined in step 1.~in place of $\bar \rho^{*}$, and to \eqref{eq: estimee bar W}.
 
 {3.} By the same arguments as in Lemma \ref{lem: jump ergodic convergence and verification}, if $(\tilde \Wf,\tilde \rho )\in (\Cc^{2} \cap \Cc^{0}_{{\rm lin}} )\x \R$ solves \eqref{eq: HJB ergodic jump}   then $\tilde \rho=\bar \rho^{*}$. In particular, $\bar \rho^{*}$ is constant and $c=\bar \rho^{*}$ by step  2.
  
  {4.} The fact that there exists  $(T_{n})_{n\ge 1}$ going to $+\infty$ such that  $(T_{n}^{-1}\bar V_{T_n}(0,\cdot))_{n\ge 1}$ converge  uniformly on compact sets to   $\bar \rho^{*}(0)$ then follows from the same arguments as in Lemma \ref{lemma: equivalence of ergodic limits}.
 \end{proof}

\subsection{First order approximation guarantees}

We  can now  turn to the main part of this paper and quantify the approximation error due to passing to the diffusive limit in the original pure jump problem. We will show below that it controlled by the H\"older regularity of ${\D}^{2} \bar \Wf$, namely that the approximation error is of the order of $\eps^{\frac{\gamma}{2}}$ for all $\gamma \in (0,1)$. In Section \ref{sec: correction term}, we will see that it can be improved by considering appropriate correction terms.  

The cornerstone of the analysis is the residual term of a second order Taylor expansion of $\bar \Wf$ performed on the Dynkin operator of the jump diffusion process \eqref{eq: def X jump}, namely: 
\begin{align}
\delta r_{\ve}(x,a):=
 \frac1\ve\int_{\R^{d'}} \left[\bar \Wf(x+b_{\ve}(x,a,e))-\bar \Wf(x)\right]\nu(\de e)-\D \bar \Wf(x)  \mu(x, a)-\frac12 \Tr[{\sigma\sigma^{\top}}(x) \D^{2}\bar \Wf(x)]\,,\label{eq: def delta r eps}
\end{align}
defined for $(x,a)\in \R^{d}\x \Ab$. The function $\delta r^\ve$ measures the error of the diffusion approximation explicitely in terms of the control problem, and thus will be shown to effectively control the error in all quantities of interest. Leveraging the regularity results in \eqref{eq: estimee bar W}, the H\"older regularity of $\D^2\bar \Wf$ yields Proposition \ref{prop: error order of delta r epsilon}, which in turn yields Theorem \ref{thm: delta epsilon bound on discounted value difference}.

\begin{Proposition}\label{prop: error order of delta r epsilon}
  Let Assumptions \ref{asmp: diffusion limit existence} and \ref{assumption: Lipschitz value discounted jump + borne EXq} hold with $p_{\xi}\ge 3$. {Fix $\gamma\in (0,1)$.} Then, there exists $L^{\gamma,1}_{\delta r} ,L^{\gamma,2}_{\delta r} >0$ such that, for each $0<\eps\le\eps_{\circ}:=  (L_{b_{1},b_{2}} )^{-2}$ and $(x,a)\in \R^{d}\x \Ab$,  
\begin{align}\label{eq: borne delta r epsilon dans thm} 
|\delta r_{\ve} (x,a)|\le  \ve^{\frac{{{\gamma}}}{2}}L^{\gamma,1}_{\delta r} (1+|x|^{3}), 
\end{align}
 and
\begin{align}\label{eq: borne EClambda K eps le long X} 
\sup_{t\ge 0}\sup_{\alpha \in \Ac} \E[\vert\delta r_{\ve}(X^{x,\alpha}_{t},\alpha_t)\vert]\le  \ve^{\frac{{{\gamma}}}{2}} L^{\gamma,2}_{\delta r}(1+|x|^{3})\;.
\end{align}
\end{Proposition}

\begin{proof} {1. } We first prove the estimate  \eqref{eq: borne delta r epsilon dans thm} using  \eqref{eq: estimee bar W}. Namely, 
\begin{align*}
\bar \Wf(x+b_{\ve}(x,a,e))-\bar \Wf(x)
&= \bar \Wf(x+\eps b_{1}(x,a,e)+\eps^{\frac12}b_{2}(x,e))- \bar \Wf(x+\eps^{\frac12}b_{2}(x,e)))\\
&+ \bar \Wf(x+\eps^{\frac12}b_{2}(x,e))-\bar \Wf(x)
 \end{align*}
 where
 \begin{align*}
 &\bar \Wf(x+\eps b_{1}(x,a,e)+\eps^{\frac12}b_{2}(x,e))- \bar \Wf(x+\eps^{\frac12}b_{2}(x,e)))\\
 &= \eps \D  \bar \Wf(x+\eps^{\frac12}b_{2}(x,e)) b_{1}(x,a,e)
 +\int_{0}^{1} \frac{\eps^{2}}{2}b_{1}(x,a,e)^{\top}\D^{2}  \bar \Wf(\hat x^{\eps}_{1}(u))b_{1}(x,a,e)\de u
 \end{align*}
 in which 
 $$
 \hat x^{\eps}_{1}(u):=x+\eps^{\frac12}b_{2}(x,e)+u \eps b_{1}(x,a,e)
 $$
 is such that 
 $$
 \sup_{u\in [0,1]}|\hat x^{\eps}_{1}(u)|\le \abs{x}+ \eps^{\frac12} L_{b_{1},b_{2}}+\eps L_{b_{1},b_{2}}(1+|x|), 
 $$
 by definition of $L_{b_{1},b_{2}}$ in Assumption \ref{asmp: diffusion limit existence}.
By \eqref{eq: estimee bar W} and  Assumption \ref{asmp: diffusion limit existence}, this implies that 
\begin{align*}
&\abs{\frac{\eps^{2}}{2}b_{1}(x,a,e)^{\top}\D^{2}  \bar \Wf(\hat x^{\eps}_{1}(u))b_{1}(x,a,e)}\\
&\le \frac{\eps^{2}}{2} (L_{b_{1},b_{2}})^{2}(1+|x|)^{2}L^{\gamma}_{\bar \Wf}(1+ |x|+\eps^{\frac12} L_{b_{1},b_{2}}+\eps L_{b_{1},b_{2}}(1+|x|)).
 \end{align*}
 {Moreover,  since $\eps^{\frac12}L_{b_{1},b_{2}}\le 1$, we have 
 $$
| \D  \bar \Wf(x+\eps^{\frac12}b_{2}(x,e))-\D  \bar \Wf(x)|\le L^{\gamma}_{\bar \Wf}(1+|x|)\eps^{\frac12}L_{b_{1},b_{2}}
 $$
 by (i) of Assumption \ref{asmp: diffusion limit existence} and \eqref{eq: estimee bar W}.}
 
 Using (ii) of Assumption \ref{asmp: diffusion limit existence}, we next obtain that  
 \begin{align*}
 \int_{\R^{d'}} \{\bar \Wf(x+\eps^{\frac12}b_{2}(x,e))-\bar \Wf(x)\} \nu(\de e)
 &=  \int_{\R^{d'}} \int_0^1\frac{\eps}{2}b_{2}(x,e)^{\top}\D^{2}  \bar \Wf(\hat x^{\eps}_{2}(u,e))b_{2}(x,e)   \de u\,\nu(\de e) 
 \end{align*}
 in which 
  $$
 \hat x^{\eps}_{2}(u,e):=x+u\eps^{\frac12}b_{2}(x,e)\in B_{1}(x)
 $$
 since $\eps^{\frac12}L_{b_{1},b_{2}}\le 1$ by assumption {and (i) of Assumption \ref{asmp: diffusion limit existence}}. Then, by \eqref{eq: estimee bar W} again  and (iii) of Assumption \ref{asmp: diffusion limit existence}  
 \begin{align*}
& \abs{\int_{\R^{d'}} \{\bar \Wf(x+\eps^{\frac12}b_{2}(x,e))-\bar \Wf(x)\} \nu(\de e)
 -   \frac{\eps}{2} \Tr[\sigma\sigma^{\top}(x) \D^{2} \bar \Wf(x)] }
 \\
 &=  \abs{\int_{\R^{d'}} \{\bar \Wf(x+\eps^{\frac12}b_{2}(x,e))-\bar \Wf(x)\} \nu(\de e)
 -  \int_{\R^{d'}} \frac{\eps}{2}b_{2}(x,e)^{\top}\D^{2}  \bar \Wf(x)b_{2}(x,e)  \nu(\de e)}
 \\
 &\le 
 \frac{\eps}{2}(L_{b_{1},b_{2}})^{2}L^{{\gamma}}_{\bar \Wf}(1+|x|)(\eps^{\frac12} L_{b_{1},b_{2}})^{\gamma}.
 \end{align*}
 
 The estimate \eqref{eq: borne delta r epsilon dans thm} is obtained by combining the above. 
 
 {2.} The estimate   \eqref{eq: borne EClambda K eps le long X}  {follows from  \eqref{eq: borne delta r epsilon dans thm},  Lemma \ref{lemma: borne EXq} and the fact that $ p_{\xi}\ge 3$. }
\end{proof}

We are now in position to state the main result of this section. 

\begin{Theorem}\label{thm: delta epsilon bound on discounted value difference}
Let Assumptions \ref{asmp: diffusion limit existence} and \ref{assumption: Lipschitz value discounted jump + borne EXq} hold with $p_{\xi}\ge 3$. Then,   {for all $\gamma\in (0,1$), there exists $L^{\gamma}_{\delta \rho}>0$} such that 
\begin{align*}
|\bar \rho^{*}-\rho^{*}_{\eps}|\le   \eps^{\frac{{\gamma}}{2}} L^{{\gamma}}_{\delta \rho}\,\mbox{ for all $\eps\in (0,1)$.}
\end{align*}
Moreover, there exists a measurable map $\hat {\rm a}: \R^{d}\mapsto \Ab$ such that 
$$
\bar \Lc^{\hat {\rm a}}\bar \Wf+r(\cdot,\hat {\rm a})=\sup_{\bar a\in\Ab}\left\{\bar \Lc^{\bar a}\bar \Wf+r(\cdot,\bar a)\right\},\;\mbox{ on $\R^{d}$}
$$
and 
\[ 
\rho_\ve^{*}- \eps^{\frac{{\gamma}}{2}} L^{{\gamma}}_{\delta \rho}\le \liminf_{T\to \infty} \frac{1}{\eta_\ve T }\EE\left[\int_0^T  r(X^{\hat {\rm a}}_{t-},\hat {\rm a}(X^{\hat {\rm a}}_{t-}))\de N_{t} \right]\,,\;\mbox{ for all $\eps\in (0,1)$,}
\]
in which $X^{\hat {\rm a}}$ solves 
$$
X^{\hat {\rm a}}_\cdot=   \int_0^\cdot \int_{\R^{d'}} b_\ve(X^{\hat {\rm a}}_{s-},\hat {\rm a}(X^{\hat {\rm a}}_{s-}),e)N(\de e,\de  s)\,.
$$
\end{Theorem}

\begin{proof}  Fix $\gamma\in (0,1)$. Hereafter, we denote by $\Wf_{\eps}$ the function $\Wf$ introduced in Theorem \ref{thm: jump ergodic rho HJB + verif} for  $\eta=\eta_{\eps}=\eps^{-1}$. By Theorems \ref{thm: jump ergodic rho HJB + verif} and \ref{thm: viscosity solution of ergodic diffusion PDE}, $\Delta^{\eps}:=\bar \Wf-\Wf_{\eps}$ solves 
  \begin{align*} 
\bar \rho^{*}- \rho^{*}_\ve&\le \sup_{a\in\Ab}\left\{  \frac1\ve \int_{\R^{d'}}\left[  \Delta^{\eps}(\cdot+b_\ve(\cdot,a,e))-  \Delta^{\eps}\right]\nu(\de e)  {-}\delta r_{\ve} (\cdot,a)\right\},\;\mbox{ on } \R^{d}.
  \end{align*}
By the same arguments as in the proof of Lemma \ref{lem: jump ergodic convergence and verification}, \eqref{eq: borne EClambda K eps le long X}  applied with $x=0$, \eqref{eq: estimee bar W}, \eqref{eq: borne lin W} and Lemma \ref{lemma: borne EXq}, we deduce that 
$$
\bar \rho^{*}- \rho^{*}_\ve\le L^{1}_{\delta \rho} \eps^{\frac{{\gamma}}{2} }
$$
for some $L^{1}_{\delta \rho}>0$ that does not depend on $\eps\in (0,1)$. Replacing $ \Delta^{\eps}$ by $- \Delta^{\eps}$ in this argument  implies that 
$$
 \rho^{*}_\ve- \bar\rho\le L^{2}_{\delta \rho} \eps^{\frac{{\gamma}}{2} }
$$
for some $L^{2}_{\delta \rho}>0$ that does not depend on $\eps\in (0,1)$. 

The second assertion of the Theorem is then proved by following the arguments in the first part of the proof of  Lemma \ref{lem: jump ergodic convergence and verification} and using the above.
\end{proof}
 
%
 
\subsection{Higher order expansions}\label{sec: correction term}

 Under additional conditions, one can exhibit a first order correction term to improve the convergence speed in Theorem \ref{thm: delta epsilon bound on discounted value difference}. It is in the spirit of the correction term introduced in \cite[Section 3.5]{abeille2021diffusive} but is formulated differently. In particular, the function $ \delta \bar \Wf_{\eps}$ introduced below depends on $\eps$ and the optimization in \eqref{eq: pde delta bar W} is performed over the whole set $\Ab$. This approach can  be iterated to higher order correction terms in an obvious manner, upon additional regularity conditions, without considering a coupled system of PDEs as in \cite[Section 3.6]{abeille2021diffusive}. 
 
 From now on, we  assume the following.
 \begin{Assumption}\label{asmp: first order expansion}  There exists $\gamma_{\circ}\in (0,\gamma]$ and $(\delta \gamma, \delta C)\in (0,1)\x \R$ such that, for each $\eps\in (0,1)$, 
we can find  $  \delta \bar \rho^{*}_{\eps}\in \R$ and $  \delta \bar \Wf_{\eps}\in \Cc^{0}_{{\rm lin}}$ satisfying $\|  \delta \bar \Wf_{\eps}\|_{\Cc^{2,\delta \gamma}_{b}(B_{1}(x))}\le \delta C(1+|x|)$ for all $x\in \R^{d}$ and
 \begin{align}\label{eq: pde delta bar W} 
\delta \bar \rho_{\eps}^{ *}=\sup_{\bar a\in \Ab} \left[  \bar \Lc^{\bar a}  \delta \bar \Wf_{\eps}+ \eps^{-\frac{{\gamma_{\circ}}}2}  [\delta r_{\eps}+f](\cdot,\bar a)\right]\mbox{ on $\R^{d}$,}
 \end{align}
 in which 
\begin{align*} 
 f(\cdot,\bar a)&:=\bar \Lc^{\bar a}   \bar \Wf  +r(\cdot,\bar a)-\bar \rho^{*}.
 \end{align*}
 
 \end{Assumption}
 
 \begin{Theorem}\label{Thm : first correction term}  Let  the conditions of Theorem \ref{thm: delta epsilon bound on discounted value difference} and Assumption \ref{asmp: first order expansion} hold.  Assume further that $p_{\bar X}\ge 3$. Then, 
 \[
 \limsup_{\ve\downarrow 0} |\delta \bar \rho_{\eps}^{*}|<\infty 
 \]
 and 
 $$
  \bar \rho^{*(1)}_{\eps}:=\bar \rho^{*}+\eps^{\frac{{\gamma_{\circ}}}2}  \delta \bar \rho_{\eps}^*,\;\eps\in (0,1), 
 $$
 satisfies 
 $$
 \limsup_{\eps\downarrow 0} \eps^{-\frac{\gamma_{\circ}+\delta \gamma}{2}}|\rho^{*}_{\eps}-  \bar \rho^{*(1)}_{\eps}|<\infty.
 $$
 Moreover, for each $\eps \in (0,1)$,  there exists a measurable map $\hat {\rm a}_{\eps}: \R^{d}\mapsto \Ab$ such that 
$$
 \bar \Lc^{\hat {\rm a}_{\eps}}  \delta \bar \Wf_{\eps}+ \eps^{-\frac{{\gamma_{\circ}}}2}  [\delta r_{\eps}+f](\cdot,\hat {\rm a}_{\eps})   =\sup_{\bar a\in\Ab}\left[  \bar \Lc^{\bar a}  \delta \bar \Wf_{\eps}+ \eps^{-\frac{{\gamma_{\circ}}}2}  [\delta r_{\eps}+f](\cdot,\bar a)\right]\;\mbox{ on $\R^{d}$}
$$
and  
\[ 
\limsup_{\eps\downarrow 0} \eps^{-\frac{\gamma_{\circ}+\delta \gamma}{2}}|\rho_\ve^{*}-\rho_{\eps}(0,\hat {\rm a}_{\eps}(X^{\hat {\rm a}_{\eps}}_{\cdot-}))|  <\infty,
\]
in which $X^{\hat {\rm a}_{\eps}}$ solves 
$$
X^{\hat {\rm a}_{\eps}}=   \int_0^\cdot \int_{\R^{d'}} b_\ve(X^{\hat {\rm a}_{\eps}}_{s-},\hat {\rm a}_{\eps}(X^{\hat {\rm a}_{\eps}}_{s-}),e)N(\de e,\de  s)\,.
$$
 \end{Theorem}

 \begin{proof}  
 It follows from the same arguments as in Lemma \ref{lem: jump ergodic convergence and verification}  and the fact that $f\le 0$ by \eqref{eq: HJB ergodic diffusive thm}  that 
 $$
 \delta \bar \rho_{\eps}^*=\sup_{\bar \alpha \in \bar \Ac}\lim_{T\to \infty}\frac1T\E\left[\int_{0}^{T}\eps^{-\frac{\gamma_{\circ}}2}[\delta r_{\eps}+f](\bar X^{0,\bar \alpha}_{s},\bar \alpha_{s})\de s \right]\le\sup_{\bar \alpha \in \bar \Ac}\limsup_{T\to \infty}\frac1T\E\left[\int_{0}^{T}\eps^{-\frac{\gamma_{\circ}}2}\delta r_{\eps}(\bar X^{0,\bar \alpha}_{s},\bar \alpha_{s})\de s \right].
 $$
Let $\hat {\rm a}$ be as in Theorem \ref{thm: delta epsilon bound on discounted value difference}. Then, $f(\cdot,\hat {\rm a})=0$ by \eqref{eq: HJB ergodic diffusive thm}. Hence,   
  $$
 \delta \bar \rho_{\eps}^* \ge \liminf_{T\to \infty}\frac1T\E\left[\int_{0}^{T}\eps^{-\frac{\gamma_{\circ}}2}\delta r_{\eps}(\bar X^{\hat {\rm a}}_{s},\hat {\rm a}(\bar X^{\hat {\rm a}}_{s}))\de s \right]
 $$
 in which $\bar X^{\hat {\rm a}}$ solves 
 $$
\bar X^{\hat {\rm a}}=   \int_0^\cdot \mu(\bar X^{\hat {\rm a}}_{s},\hat {\rm a}(\bar X^{\hat {\rm a}}_{s})) \de  s 
+ \int_0^\cdot \sigma(\bar X^{\hat {\rm a}}_{s}) \de  W_{s}  .
$$
Note that the existence of a solution of the above is guaranteed, upon considering another probability space and Brownian motion. Combining the above inequalities with \eqref{eq: borne delta r epsilon dans thm}, the fact that $\gamma_{\circ}\le \gamma$, and the second assertion of Assumption \ref{assumption: Lipschitz value discounted jump + borne EXq} with $p_{\bar X}\ge 3$ shows that $ |\delta \bar \rho_{\eps}^*|\le C'$ for some $C'>0$ that does not depend on $\eps\in (0,\eps_{\circ}]$. 

Moreover, by Assumption \ref{asmp: first order expansion} and the same arguments as in the proof of Proposition \ref{prop: error order of delta r epsilon}, 
$$
 \delta r_{\ve}^{'} (x,a):=
\frac1\ve\int_{\R^{d'}} \left[\delta \bar \Wf_{\eps}(x+b_{\ve}(x,a,e))-\delta \bar \Wf_{\eps}(x)\right]\nu(\de e)-\bar\Lc^{a} \delta \bar \Wf_{\eps} 
$$
satisfies 
$$
|\delta r^{'}_{\ve}  (x,\cdot)|\le  \ve^{\frac{\delta\gamma}{2}}C'' (1+|x|^{3}),\;x \in \R^{d}, 
$$
for some $C''>0$ that does not depend on $\eps\in (0,\eps_{\circ}]$. Since, by construction, $\bar \Wf^{(1)}_{\eps}:=\bar \Wf+\eps^{\frac{\gamma_{\circ}}2} \delta \bar \Wf_{\eps}$ solves 
$$
 \bar \rho^{*(1)}_{\eps}=\sup_{  a\in \Ab} \left[\frac1\eps \int_{\R^{d'}} \left[\bar \Wf^{(1)}_{\eps}(\cdot+b_{\eps}(\cdot,a))-\bar \Wf^{(1)}_{\eps}\right]\nu(\de e)- \eps^{\frac{\gamma_{\circ}}2}\delta r^{'}_{\ve}(\cdot,a)+r(\cdot,a)\right]\mbox{ on $\R^{d}$,}
$$
 the same arguments as in the proof of Theorem \ref{thm: delta epsilon bound on discounted value difference} then imply that $| \bar \rho^{*(1)}_{\eps}-\rho^{*}_{\eps}|\le L \eps^{\frac{\gamma_{\circ}+\delta\gamma}{2}}$, for some $L>0$ that does not depend on $\eps \in (0,\eps_{\circ}]$, and also lead to the last assertion of the Theorem. 
 
 \end{proof}

\section{Numerical resolution of the ergodic diffusive problem}\label{sec: numerical resolution}


The numerical resolution of \eqref{eq: HJB ergodic diffusive thm} can be done by using standard finite difference schemes as explained in \cite[Chapter 7]{kushner2001numerical}. We focus on the one-dimensional case $d=1$ for simplicity, and also because similar schemes in higher dimension often have to be constructed on a case by case basis, see e.g.~ \cite[Chapter 5]{kushner2001numerical}. 

Given $\kappa \in \N$, $\kappa\ge3$, and $h>0$, we consider the space grid $\Mc^{\kappa}_h:=\{ z_{i}:=-\kappa h+(i-1)h, \;1\le i\le 2\kappa+1\}$. We use the notation $\overset{\circ}{\Mc^{\kappa}_h}:=\Mc^{\kappa}_h\setminus\{z_{1},z_{2\kappa+1}\}$ and denote by ${\rm L}^{\kappa}_{h}$   the collection of real-valued maps $\vp$ defined on $\Mc^{\kappa}_{h}$. For    $\vp  \in {\rm L}^{\kappa}_{h}$, we define the usual finite (central) differences operators: 
$$
\Delta_{h} \vp(x):=\frac{\vp(x+h)-\vp(x-h)}{2h},\;\Delta^{2}_{h} \vp(x):=\frac{\vp(x+h)+\vp(x-h)-2\vp(x)}{h^{2}}, \;x\in \overset{\circ}{\Mc^{\kappa}_h},
$$
and set 
\begin{align}\label{eq: schema diff fini initial} 
\bar \Lc^{\bar a}_{h}\vp:=\mu(\cdot,\bar a)\Delta_{h} \vp+\frac12 \sigma^{2}\Delta^{2}_{h} \vp,\;\bar a \in \Ab.
\end{align}
Then, we  approximate the solution $(\bar \rho^{*},\bar \Wf)$ of \eqref{eq: HJB ergodic diffusive thm} by a solution   $(\bar \rho^{\kappa,*}_{h},\bar \Wf^{\kappa}_{h})\in \R\times {\rm L}^{\kappa}_{h}$ of 
\begin{align}  
\bar \rho^{\kappa,*}_{h}&=\sup_{\bar a \in \Ab} \left\{\bar \Lc^{\bar a }_{h}	\bar \Wf^{\kappa}_{h} +r(\cdot,\bar a)\right\},\mbox{ on }	\overset{\circ}{\Mc^{\kappa}_h},	\label{eq: edp diff finie}
\end{align}
with a suitable reflecting boundary at $z_{1}$ and $z_{2\kappa+1}$, see below. Note that $\bar \Wf^{\kappa}_{h}$ is defined only up to a constant, and that we can, and will, set $\bar \Wf^{\kappa}_{h}(0)=\bar \rho^{\kappa,*}_{h}$ in the following. 
Let us now denote by  $\Ar$  the collection of measurable maps from $\R$ to $\Ab$ and   identify, given $\bar{\rm a}\in \Ar$, $\bar \Wf^{\kappa}_{h}$ and $r(\cdot,\bar {\rm a}(\cdot))$ on $\Mc^{\kappa}_h$	 to column vectors $\bar {\mathscr W}^{\kappa}_{h}:=(\bar \Wf^{\kappa}_{h}(z_{i}))_{1\le i\le 2\kappa+1}$ and ${\mathscr R}(\bar {\rm a}):=(r(z_{i},\bar {\rm a}(z_{i}))_{1\le i\le 2\kappa+1}$ of $\R^{2\kappa+1}$.  Then, to solve \eqref{eq: edp diff finie} on $ {\Mc^{\kappa}_h}$ with $\bar \Wf^{\kappa}_{h}(0)=\bar \rho^{\kappa,*}_{h}\Delta t_{h}$, including a suitable reflection term on the boundary $\{z_{1},z_{2\kappa+1}\}$, we search for $(\bar \rho^{\kappa,*}_{h},\bar {\mathscr W}^{\kappa}_{h})\in \R\x \R^{2\kappa+1}$ that satisfies 
\begin{align}  
\bar {\mathscr W}^{\kappa}_{h}&=\sup_{\bar {\rm a} \in \Ar}\bar Q^{\bar {\rm a} }_{h}	 \left\{\bar {\mathscr W}^{\kappa}_{h}-e\bar \rho^{\kappa,*}_{h}\Delta t_{h}  +{\mathscr R}(\bar {\rm a})\Delta t_{h}\right\},\mbox{ on }	 {\Mc^{\kappa}_h}	 \label{eq: edp diff finie matriciel}\\
\bar \Wf^{\kappa}_{h}(0)&=\bar \rho^{\kappa,*}_{h} \Delta t_{h}  \label{eq: edp diff finie matriciel normalisation rho}
\end{align}
 where  
 $$
 \Delta t_{h}:=\frac{h^{2}}{(L_{b_{1},b_{2}})^{2}},
 $$ 
 $e$ is the column vector of $\R^{2\kappa+1}$ with all entries equal to $1$, and $\bar Q^{\bar {\rm a} }_{h}=((\bar Q^{\bar {\rm a} }_{h})^{i,j})_{1\le i,j\le 2\kappa+1}$ is the  matrix with all entries null except for     
\[(\bar Q^{\bar  {\rm a} }_{h})^{i,i-1}:= q^{-}_{h}(z_{i} ,\bar{\rm a}(z_{i}))\,,\;(\bar Q^{\bar  {\rm a} }_{h})^{i,i}:= q_h(z_i, \bar{\rm a}(z_i))\,,\;\mbox{ and }\;(\bar Q^{\bar  {\rm a}  }_{h})^{i,i+1}:=q^{+}_{h}(z_{i},\bar{\rm a}(z_{i}))\;, \]
for $1<i<2\kappa+1$, with
$$
q_{h}:=1-\frac{\sigma^{2}}{(L_{b_{1},b_{2}})^{2}},\;q^{+}_{h} :=\frac{  \mu h+\sigma^{2} }{2(L_{b_{1},b_{2}})^{2}}\,, \;\mbox{ and }\; q^{-}_{h} :=\frac{  -\mu h+\sigma^{2} }{2(L_{b_{1},b_{2}})^{2}},
$$
and except for
\begin{align*}
(\bar Q^{\bar {\rm a} }_{h})^{1,j}&:=(\bar Q^{\bar {\rm a} }_{h})^{3,j} \mbox{ for } j=2,3,4\\
(\bar Q^{\bar {\rm a} }_{h})^{2\kappa+1, j}&:=(\bar Q^{\bar {\rm a} }_{h})^{2\kappa-1,j} \mbox{ for } j=2\kappa-2,2\kappa-1,2\kappa.
 \end{align*}
The above scheme is of the form of \cite[Chapter 7 (2.3)]{kushner2001numerical}.

Without loss of generality, one can assume from now on that 
 $$
 L_{b_{1},b_{2}}>\|\sigma\|_{\Cc^{0}_{b}}.
 $$
Then, recalling (i)-(ii) of Assumption \ref{asmp: diffusion limit existence}, $\bar Q^{\bar {\rm a} }_{h}$ defines a transition probability matrix satisfying 
$$
\min_{1\le i\le 2\kappa+1}\min_{1\vee(i-1)\le j\ne i\le (2\kappa +1)\wedge (i+1)} (\bar Q^{\bar  {\rm a} }_{h})^{i,j}=:\underline{p}_{h}>0  
$$
whenever 
\begin{align}\label{eq: cond kappa h2 le demi - L} 
L_{b_{1},b_{2}}(1+\kappa h)h<\varsigma.
\end{align}
Given $\bar{\rm a}\in\Ar$, let  $(Z^{x,\bar{\rm a}}_{t})_{t\in \N}$ be the Markov chain starting from $x\in \Mc^{\kappa}_h$ and such that 
$$
\P[Z^{x,\bar{\rm a}}_{t+1}=z_{j}|Z^{x,\bar{\rm a}}_{t}=z_{i}]=(\bar Q^{\bar  {\rm a} }_{h})^{i,j},\;1\le i,j\le 2\kappa+1, \;t\in \N,
$$
then
\begin{align}\label{eq: p Z kappa=0} 
\P[Z^{x,\bar{\rm a}}_{\kappa}=0]\ge (\underline{p}_{h})^{\kappa}>0,
\end{align}
under \eqref{eq: cond kappa h2 le demi - L}.
Then, assuming further that 
\begin{align}\label{eq: conti r en a} 
(b,r)(x,\cdot): \Ab\mapsto \R^{2} \mbox{ is continuous for all }x\in \R,
\end{align}
it follows that the conditions of \cite[Chapter 7 Theorem 2.1]{kushner2001numerical}  hold so that  $(\bar \rho^{\kappa,*}_{h},\bar {\mathscr W}^{\kappa}_{h})$ is well-defined and can be computed by using the iterative scheme of \cite[Chapter 7 (2.3)]{kushner2001numerical}.

Under the following conditions, one can exhibit an upper-bound on the convergence rate of the above numerical scheme.

\begin{Assumption}\label{ass : contraction diffusion diff finies}
There exists a   function $\bar \xi\in \Cc^{3}(\RR)$, $p_{\bar \xi}\ge 2$, and constants $C^{1}_{\bar \xi}>0$ and $C^{2}_{\bar \xi}\in\RR$ such that 
\eqref{eq:cond suff mean revert bar V} and \eqref{eq:cond suff mean revert bar V 2} hold
   for all $x\in \RR^{d}$. 
  Moreover, there are constants $L>0$, $\Upsilon>0$, and $C_\Upsilon>0$, such that $|\D^{2}\bar \xi(x) |+|\D^{3}\bar \xi(x) |\le L(1+|x|^{p_{\bar \xi}-1})$ for all $x\in \R$, and ${\rm sgn}(x)\D\bar\xi(x)\ge C_\Upsilon \abs{x}^{p_{\bar \xi}-1}$ for all $\abs{x}\ge \Upsilon$, where ${\rm sgn}(\cdot)$ is the sign function.
\end{Assumption}

\begin{Proposition}\label{prop : conv rho num}  Let Assumptions \ref{asmp: diffusion limit existence},  \ref{assumption: Lipschitz value discounted jump + borne EXq}  and \ref{ass : contraction diffusion diff finies} hold with $p_{\xi}\ge 3$.  Assume further that \eqref{eq: conti r en a} is satisfied. Then, there exists $L_{\rm num}>0$ and $h_{\rm num}>0$ such that, for all $(h,\kappa)\in (0,h_{\rm num})\x \N$, satisfying \eqref{eq: cond kappa h2 le demi - L},  $\kappa h^{2}\le 1$ and ${(\kappa-3) h}\ge \Upsilon$, we have
\[
|\bar \rho^{\kappa,*}_{h}-  {\bar\rho^{*}}|\le L_{\rm num} (h^\gamma + h^{-1}  \abss{\kappa h}^{-|p_{\bar \xi}-1|}).
\]
 In particular, 
\begin{align*}
|\bar \rho^{\kappa,*}_{h}-\rho^{*}_{\eps}|\le L_{\rm num} (h^\gamma + h^{-1}  \abss{\kappa h}^{-|p_{\bar \xi}-1|})+   \eps^{\frac{{\gamma}}{2}} L^{{\gamma}}_{\delta \rho}\,\mbox{ for all $\eps\in (0,1)$.}
\end{align*}
\end{Proposition}

\begin{proof} Given $\bar{\rm a}\in \Ar$ and $x\in \Mc^{\kappa}_{h}$, let $\tilde X^{x,\bar{\rm a}}$ be the pure jump continuous time Markov  {chain} defined by a sequence of jump times $(\tau_{n})_{n\ge 1}$ such that the increment{s} $(\tau_{n+1}-\tau_{n})_{n\ge 0}$ (with the convention $\tau_{0}=0$) are independent and identically distributed according to the exponential law of mean $\Delta t_{h}$ and such that{, for $n\ge 1$},
$$
\P[\tilde X^{x,\bar{\rm a}}_{\tau_{n}}=z_{i}|{(\tilde X^{x,\bar{\rm a}}_{0},\tau_{0}),\ldots,(\tilde X^{x,\bar{\rm a}}_{\tau_{n-1}},\tau_{n-1}),\tau_{n}}]=  
(\bar Q^{\bar  {\rm a} }_{h})^{i,j(\tilde X^{x,\bar{\rm a}}_{\tau_{n-1}})},
$$
with 
$$
j(\tilde X^{x,\bar{\rm a}}_{\tau_{n-1}}) {\in \N\;\mbox{ s.t. }\;}z_{j(\tilde X^{x,\bar{\rm a}}_{\tau_{n-1}})}=\tilde X^{x,\bar{\rm a}}_{\tau_{n-1}},
$$
and $\tilde X^{x,\bar {\rm a}}=\tilde X^{x,\bar{\rm a}}_{\tau_{n-1}}$ on $[\tau_{n-1},\tau_{n})$.  

1. First note that, by construction, $\bar \Wf^{\kappa}_{h}$ is bounded on the finite set $\Mc^{\kappa}_{h}$. Then,  by the arguments in the proof of  Lemma \ref{lem: jump ergodic convergence and verification} and \eqref{eq: edp diff finie matriciel}, we have 
\begin{align}\label{eq: caracterisation bar rho * kappa h} 
\bar \rho^{\kappa,*}_{h}=\sup_{\bar {\rm a}\in \Ar} \lim_{T\to \infty} \frac1T\E\left[\int_{0}^{T} r(\tilde X^{x,\bar {\rm a}}_{s},\bar {\rm a}(\tilde X^{x,\bar{\rm a}}_{s}))ds\right].
\end{align}

  2. We now prove that there exists $C^{1'}_{\bar \xi},C^{2'}_{\bar \xi},h_{{\rm num}}>0$  such that, for all $x\in \R$, $\bar {\rm a}\in \Ar$, $0<h\le  h_{{\rm num}}$ and $\kappa$ such that \eqref{eq: cond kappa h2 le demi - L} holds,  $\kappa h^{2}\le 1$ and ${(\kappa-3) h}\ge \Upsilon$, we have
  \begin{align}\label{eq: control Lp tilde X } 
  \E[|\tilde X^{x,\bar{\rm a}}_{t}|^{p_{\bar \xi}}]\le C^{2}_{\bar \xi}\left\{e^{-  C^{1'}_{\bar \xi} t}C^{2}_{\bar \xi}|x|^{p_{\bar \xi}}+\frac{C^{2'}_{\bar \xi}}{  C^{1'}_{\bar \xi}}(1-e^{-  C^{1'}_{\bar \xi} t})\right\} ,\;t\ge 0.
  \end{align}
  Using Assumption \ref{ass : contraction diffusion diff finies} and Assumption \ref{asmp: diffusion limit existence}, and Taylor expansions of first and second orders, we first deduce that, for $x\in\overset{\circ}{\Mc^{\kappa}_h}$,
  \begin{align*} 
    \D\bar \xi (x)\mu(x,\bar {\rm a}(x))+\frac12 \sigma^{2}(x) \D^2\bar \xi(x) 
    &= \frac{1}{\Delta t_{h}}\E[   \bar \xi(\tilde X^{x,\bar{\rm a}}_{\tau_{1}})-\bar \xi(x)]-c(x)h,\;
  \end{align*}
  in which $|c(x)|\le C(1+|x|^{p_{\bar \xi}})\le C(1+C^{2}_{\bar \xi}\bar \xi(x))$ for some $C>0$ independent on $x\in \R$, $\bar {\rm a}\in \Ar$, $h$ and $\kappa$.
  Using   \eqref{eq:cond suff mean revert bar V}, this implies that, for $x\in\overset{\circ}{\Mc^{\kappa}_h}$,
  \begin{align}\label{eq: b1}
  \frac{1}{\Delta t_{h}}\E[   \bar \xi(\tilde X^{x,\bar{\rm a}}_{\tau_{1}})-\bar \xi(x)]\le  -\left(C^{1}_{\bar \xi} -hCC^{2}_{\bar \xi}\right)\bar \xi(x)+C^{2}_{\bar \xi}+Ch .
  \end{align}
   Consider now the case $x=z_1$, the other boundary being symmetric. Let $\Xi$ be a discrete random variable taking value $k\in\{1,2,3\}$ with probability $(\bar Q^{\bar{ \rm a}}_h)^{1,k}$. Using Assumption \ref{ass : contraction diffusion diff finies} and \eqref{eq:cond suff mean revert bar V 2},  we obtain that, for some  random variable $\hat z_\Xi$ such that $\hat z_\Xi\in [z_1,z_1+\Xi h]$ a.s.,
  \begin{align} \frac{1}{\Delta t_h}\EE[\xi(z_1+\Xi h)-\xi(z_1)]&=\frac{1}{\Delta t_h}\EE[\Xi h\D\xi(\hat z_\Xi)]\le -\frac {L_{b_1,b_2}^2}{h}C_\Upsilon\EE[\Xi \abs{\hat z_\Xi}^{p_{\bar\xi}-1}]\nonumber\\
     &\le  - L_{b_1,b_2}^2 C_\Upsilon\EE[ \kappa h \abs{\hat z_\Xi}^{p_{\bar\xi}-1}]\le -C'\bar \xi(z_1)\label{eq: b2} 
    \end{align}
    when $\abs{(\kappa-3) h}\ge\Upsilon$  and $\kappa h^{2}\le 1$, in which $C'>0$ does not depend on $\kappa$ nor $h$. The above also holds with $z_{2\kappa+1}$ in place of $z_{1}$. Combining \eqref{eq: b1}-\eqref{eq: b2},  we obtain
    \begin{align*} 
      \frac{1}{\Delta t_{h}}\E[ \bar \xi(\tilde X^{x,\bar{\rm a}}_{\tau_{1}})-\bar \xi(x)] \le -\left((C^{1}_{\bar \xi} -hCC^{2}_{\bar \xi})\wedge C'\right)\bar \xi(x)+C^{2}_{\bar \xi}+Ch \le -  C^{1'}_{\bar \xi}\bar \xi(x)+C^{2'}_{\bar \xi},
     \end{align*}
     for all $x\in\Mc^\kappa_h$, whenever $h\le h_{\rm num}$, in which $C^{1'}_{\bar\xi},C^{2'}_{\bar\xi}, h_{\rm num}>0$ do not depend on $\kappa$ nor $h$. One can then argue as in the proof of Lemma \ref{lemma: borne EXq} to obtain \eqref{eq: control Lp tilde X }.

 3. From now on, we denote by $C>0$ a generic constant, which may change from line to line, but does not depend on $\kappa$ or $h$. We now appeal to \eqref{eq: estimee bar W} and the Lipschitz continuity of $(\mu,\sigma^{2})$, and use the fact that $h\le 1$ to deduce by consistency arguments  that, for $x\in {\Mc^{\kappa}_h}$, 
 $$
\bar \Lc^{\bar {\rm a}(x)} \bar \Wf(x) = \frac{1}{\Delta t_{h}}\E[ \bar \Wf(\tilde X^{x,\bar{\rm a}}_{\tau_{1}})-\bar \Wf(x)] +\delta r_{h}(x,{\rm a}(x))
 $$
in which 
\begin{align*} 
  |\delta r_{h}(x,\bar {\rm a}(x))|\le& C((1+|x|)h+h^{\gamma})(1+|x|) +Ch^{-1}(1+\abs{x})\1_{\{|x|= \kappa h  \}} .
  \end{align*}
The above  combined with \eqref{eq: HJB ergodic diffusive thm} implies that   
\begin{align*}
\bar \rho^{*}=\frac1{\Delta t_{h}}\sup_{\bar a\in\Ar}\E\left[ \bar \Wf(\tilde X^{x,\bar a}_{\tau_{1}})-\bar \Wf(x) + \{r(x,\bar a)+\delta r_{h}(x,\bar a)\}\Delta t_{h}   \right]\;\mbox{ for } x\in  \Mc^{\kappa}_h.
\end{align*} 
Arguing again as in the proof of  Lemma \ref{lem: jump ergodic convergence and verification},  recalling \ref{eq: caracterisation bar rho * kappa h}, and combining \eqref{eq: control Lp tilde X } with H\"older's and Markov's inequality, we deduce that we can find $C,C',C''>0$ such that
\begin{align*}
|\bar \rho^{\kappa,*}_{h}-\bar \rho^{*}|
\le& \sup_{\bar{\rm a}\in \Ar}\limsup_{T\to \infty} \frac1T \E\left[\int_{0}^{T} |\delta r_{h}|(\tilde X^{0,\bar{\rm a}}_{s},\bar{\rm a}(\tilde X^{0,\bar{\rm a}}_{s}))\left(\1_{\{|\tilde X^{0,\bar{\rm a}}_{s}|<  \kappa h\}} +  \1_{\{|\tilde X^{0,\bar{\rm a}}_{s}|= \kappa h\}} \right)ds \right]\\
\le &\sup_{\bar{\rm a}\in \Ar}\limsup_{T\to\infty}\frac CT\int_0^T\EE\bigg[(h+h\abss{\tilde X^{0,\bar{\rm a}}_{s}}^2+h^\gamma\abss{\tilde X^{0,\bar{\rm a}}_{s}}) + h^{-1}(1+\abss{\tilde X^{0,\bar{\rm a}}_{s}})\1_{\{\abs{\tilde X^{0,\bar{\rm a}}_{s}}=\kappa h\}}\bigg]\de s \\
\le& C'(h+h^\gamma) + \sup_{\bar{\rm a}\in \Ar}\limsup_{T\to\infty} h^{-1}\frac CT\int_0^T \EE[(1+\abss{\tilde X^{0,\bar{\rm a}}_{s}})^{p_{\bar \xi}}]^{\frac1{p_{\bar \xi}}}\left(\frac{\EE[\abss{\tilde X^{0,\bar{\rm a}}_{s}}^{p_{\bar \xi}}]}{(\kappa h)^{p_{\bar \xi}}}\right)^{\frac{p_{\bar \xi}-1}{p_{\bar \xi}}}\de s\\
\le& C''(h^\gamma + h^{-1}  \abss{\kappa h}^{-|p_{\bar \xi}-1|})\,.
\end{align*}
It remains to appeal to Theorem \ref{thm: delta epsilon bound on discounted value difference} to complete the proof. 
\end{proof}

One can also construct from the above scheme an almost optimal control for the original pure jump problem. 
For this purpose, let $\phi$ be a smooth density function with support $(-1,1)$ such that $\|\phi\|_{\Cc^{2}_{b}}\le 1$. Given $n\ge 1$, let 
$$
\bar {  \Wf}^{\kappa,n}_{h}(x):=\int (\bar {  \Wf}^{\kappa}_{h}(y)-\bar\rho^{\kappa,*}_h\Delta t_{h})\phi(n(y-x))\de y,\;x\in \R,
$$ 
with the convention that $\bar {\Wf}^{\kappa}_{h}=\bar { \Wf}^{\kappa}_{h}(z_{1})-\bar\rho^{\kappa,*}_h\Delta t_{h}$ on $(-\infty,z_{1})$ and $\bar {  \Wf}^{\kappa}_{h}=\bar {  \Wf}^{\kappa}_{h}(z_{2\kappa+1})-\bar\rho^{\kappa,*}_h\Delta t_{h}$ on $(z_{2\kappa+1},\infty)$.

Let $\bar {\rm a}^{\kappa,n}_{h}\in \Ar$ be such that 
\begin{align}
\bar {\rm a}^{\kappa,n}_{h}\in {\rm arg}\max_{a\in \Ab}[\bar \Lc^{a}\bar {  \Wf}^{\kappa,n}_{h}+r(\cdot,a)],\;\mbox{ on } \R,\label{eq: smoothed numerical max on a}
\end{align}
and set 
 $\hat \alpha^{\kappa,n}_{h}=\bar{\rm a}^{\kappa,n}_{h}(\hat X^{\kappa,n,h})$ with 
\begin{align*}
\hat X^{\kappa,n,h}_\cdot&=   \int_0^\cdot \int_{\R^{d'}} b_\ve(\hat X^{\kappa,n,h}_{s-},\bar{\rm a}^{\kappa,n}_{h}(\hat X^{\kappa,n,h}_{s-}),e)N(\de e,\de  s)\,.
 \end{align*}
The control $\bar{\rm a}_h^{\kappa,n}$ can be computed numerically at low cost, e.g.~via first order conditions; Proposition \ref{prop: conv politique num} gives the associated error bounds. This approach seems novel in the literature, and is of independent methodological interest.

\begin{Proposition}\label{prop: conv politique num} Let the conditions of Proposition \ref{prop : conv rho num} 
hold. 
Then, there exists $C>0$ such that, for all $K>0$, $n\ge 1$ and  $\eps\in (0,1)$, 
\begin{align}\label{eq: rho eps control optim schema num} 
|\rho^{*}_{\eps}(0,\hat \alpha^{\kappa,n}_{h})-\rho^{*}_{\eps}|\le C\left(  n^{-\gamma}+  \eps^{\frac{\gamma}2}+  n\sup_{x\in B_{K}(0)}  |\bar {  \Wf}^{\kappa }_{h} -\bar\rho^{\kappa,*}_h\Delta t_{h} -\bar \Wf|(x)+n K^{-1}\right).
\end{align}
If, moreover,
\\
{\rm (i)}  $\sigma$ is constant, \\
{\rm (ii)}  there exists $c>0$ such that   $\mu(x)-\mu(x')\le -c (x-x')$ if  $x\ge x' \in \R$,\\
{\rm (iii)}  there exists $R>0$ such that 
\begin{align}\label{eq: hyp R large}  
\sup_{\abs{x}>R}\sup_{\bar a \in \Ab}\mu(x,\bar a)x <-\frac12\sigma^2,
\end{align}
then
$$
  \limsup_{h \to 0} \sup_{x\in B_{K}(0)} |\bar {  \Wf}^{\kappa_{h} }_{h}-\bar\rho^{\kappa_{h},*}_h\Delta t_{h}-\bar \Wf|(x) =0
$$ 
for any family $(\kappa_{h})_{h>0}\subset (2\N+1)$ such that  $\lim_{h\downarrow 0 }\kappa_{h}h^{2}=0$ and $\lim_{h\downarrow 0 }\kappa_{h}h^{\frac{p_{\bar \xi}}{p_{\bar \xi}-1}}=\infty$.
\end{Proposition}
 
\begin{proof} 1. We first note that 
\begin{align*} 
\D\bar {  \Wf}^{\kappa,n}_{h}(x)&=\int \D \bar \Wf(y) \phi(n(y-x))dy-\int (\bar {  \Wf}^{\kappa}_{h}-\bar\rho^{\kappa,*}_h\Delta t_{h}-\bar \Wf)(y)n\phi'(n(y-x))dy\\
\D^{2}\bar {  \Wf}^{\kappa,n}_{h}(x)&=\int \D^{2} \bar \Wf(y) \phi(n(y-x))dy+\int (\bar {  \Wf}^{\kappa}_{h}-\bar\rho^{\kappa,*}_h\Delta t_{h}-\bar \Wf)(y)n^{2}\phi''(n(y-x))dy,\;x\in \R,
\end{align*}
in which $\phi'$ and $\phi''$ stand for the first and second order derivatives of $\phi$. Hence, it follows from  \eqref{eq: estimee bar W}, (i) of Assumption \ref{asmp: diffusion limit existence} and \eqref{eq: HJB ergodic diffusive thm}  that
\begin{align*}
\bar \Lc^{\bar {\rm a}^{\kappa,n}_{h}(\cdot)}\bar {  \Wf}^{\kappa,n}_{h}+r(\cdot,\bar {\rm a}^{\kappa,n}_{h}(\cdot))&=\max_{a\in \Ab}[\bar \Lc^{a}\bar {  \Wf}^{\kappa,n}_{h}+r(\cdot,a)]\\
&\ge  \max_{a\in \Ab}[\bar \Lc^{a}\bar {  \Wf} +r(\cdot,a)]-\frac12\delta r^{\kappa,n}_{h}
\\
&=\bar \rho^{*}-\frac12\delta r^{\kappa,n}_{h}
 \end{align*}
 in which $\delta r^{\kappa,n}_{h}$ satisfies, for some $C>0$ independent on $n, \kappa$ and $h$,  
$$
0\le \delta r^{\kappa,n}_{h}(x)\le C (1+|x|)\left[n^{-\gamma}+2 n^{2} \int_{B_{\frac1n}(x)} |\bar {  \Wf}^{\kappa }_{h}-\bar\rho^{\kappa,*}_h\Delta t_{h}-\bar \Wf|(y)dy \right]. 
$$
Similarly, 
\begin{align*}
\bar \rho^{*}-\frac12\delta r^{\kappa,n}_{h}
&\le \bar \Lc^{\bar {\rm a}^{\kappa,n}_{h}(\cdot)}\bar {  \Wf}^{\kappa,n}_{h}+r(\cdot,\bar {\rm a}^{\kappa,n}_{h}(\cdot))\\
&\le \bar \Lc^{\bar {\rm a}^{\kappa,n}_{h}(\cdot)}\bar {  \Wf}+r(\cdot,\bar {\rm a}^{\kappa,n}_{h}(\cdot))+\frac12\delta r^{\kappa,n}_{h}.
 \end{align*}
 Recalling \eqref{eq: def delta r eps} and Theorem \ref{thm: delta epsilon bound on discounted value difference}, we deduce that 
$$
  \rho^{*}_{\ve}-\eps^{\frac{\gamma}2} L^{\gamma}_{\delta \rho}\le   \frac1\ve\int_{\R^{d'}} \left[\bar \Wf(x+b_{\ve}(x,\bar {\rm a}^{\kappa,n}_{h}(x),e))-\bar \Wf(x)\right]\nu(\de e) +r(x,\bar {\rm a}^{\kappa,n}_{h}(x))+\delta r^{\kappa,n}_{h}(x)-\delta r_{\ve}(x,\bar {\rm a}^{\kappa,n}_{h}(x)) 
$$
for all $x\in \R$. 
We then deduce \eqref{eq: rho eps control optim schema num}  by the same arguments as in the proof of Theorem \ref{thm: delta epsilon bound on discounted value difference}.

2. It remains to prove the second assertion of the proposition.  For ease of notations, we do not write the dependence of $\kappa$ with respect to $h$, but we keep in mind that we can consider $h$ small and that $\kappa$ can be adjusted as soon as  the following results can apply to sequences such that $\lim_{h\downarrow 0 }\kappa_{h}h^{2}=0$ and $\lim_{h\downarrow 0 }\kappa_{h}h^{\frac{p_{\bar \xi}}{p_{\bar \xi}-1}}=\infty$.
\\
2.a. We first prove that  $[\bar \Wf^{\kappa}_{h}]_{\Cc^{0}_{\rm lin}(\Mc^{\kappa}_h)}$ does not depend on $\kappa$ nor $h$.  To this end, we adapt the arguments of Lemma \ref{lemma: Lipschitz value discounted jump} and Theorem \ref{thm: jump ergodic rho HJB + verif}, and actually prove that it is Lipschitz, uniformly in $\kappa$ and $h$. \\
Let $(\xi_{j})_{j\ge 1}$ be a sequence of i.i.d.~random variables following the uniform distribution on $[0,1]$ and let  $(\tau_{n})_{n\ge 1}$ be a random sequence, independent of $(\xi_{j})_{j\ge 1}$, such that the increment{s} $(\tau_{n+1}-\tau_{n})_{n\ge 0}$ (with the convention $\tau_{0}=0$) are independent and identically distributed according to the exponential law of mean $\Delta t_{h}$. 
Given $(x,\bar a,y)\in \R\x \Ab\x \R$, set 
$$
\Delta {\rm x}(x, \bar  a,y):=h\1_{\{y\le q^{+}_{h}(x, \bar  a)\}}-h\1_{\{q^{+}_{h}(x,  a)<y\le( q^{+}_{h}+q^{-}_{h})(x, \bar  a)\}}, \;\mbox{if } x\in \overset{\circ}{\Mc^{\kappa}_h},
$$
and  
\begin{align*} 
\Delta {\rm x}(z_{1}, \bar  a,y):=2h+\Delta {\rm x}(z_{3}, \bar  a,y)\;,\;\Delta {\rm x}(z_{2\kappa+1}, \bar  a,y):=-2h+\Delta {\rm x}(z_{2\kappa-1},\bar   a,y) .
\end{align*}
Let  $\check{\Ac\;\,}\!\!\!$ denote the collection of $\Ab$-valued processes that are predictable with respect to the filtration generated by $t\mapsto \sum_{i\ge 1}\xi_{i}\1_{\{\tau_{i}\le t\}}$
Given $\check \alpha\in \check{\Ac\;\,}\!\!\!$ and $x\in \Mc^{\kappa}_{h}$, let $\tilde X^{x,\check \alpha}$ be the pure jump continuous time Markov  {chain} defined by 
$$
\tilde X^{x,\check \alpha}_{\tau_{i+1}}=\tilde X^{x,\check \alpha}_{\tau_{i}}+\Delta {\rm x}(\tilde X^{x,\bar{\rm a}}_{\tau_{i}},  \check \alpha_{\tau_{i}},\xi_{i+1})
$$
and $\tilde X^{x,\check \alpha}=\tilde X^{x,\check \alpha}_{\tau_{i}}$ on $[\tau_{i},\tau_{i+1})$, $i\ge 0$. It has the same law as the process introduced at the beginning of the proof of Proposition \ref{prop : conv rho num}, and in particular 
\begin{align*}
\bar \rho^{\kappa,*}_{h}=\sup_{\check \alpha\in \check{\Ac\;\,}\!\!\!} \lim_{T\to \infty} \frac1T\E\left[\int_{0}^{T} r(\tilde X^{x,\bar{\rm a}}_{s},  \check \alpha_{s})ds\right].
\end{align*}

We  set
$$
\check V_{\lambda}(x):=\sup_{\check\alpha\in\check{\Ac\;\,}\!\!\!}\E\left[\int_{0}^{\infty}e^{-\lambda s} r(\check X^{x,\check \alpha}_{s},\check \alpha_{s})ds\right].
$$

2.a.(i) We first need to  obtain contraction estimates similar to the ones obtained in the proof  of Lemma \ref{lemma: Lipschitz value discounted jump}. We restrict for the moment to the case where the distance between the initial data are in $2h\Z$.
\\
Let us first observe that,  for $h$ small enough for condition \eqref{eq: cond kappa h2 le demi - L} to hold, we have $q^{+}_{h}(x)<(q^{+}_{h}+ q^{-}_{h})(x')=\sigma^{2}/(L_{b_{1},b_{2}})^{2}=:m$ and conversely. Although recall that, by Assumption, $ \mu(x)-\mu(x')\le -c (x-x')\le 0$ and therefore $q^{+}_{h}(x)\le q^{+}_{h}(x')$ if  $x\ge x' \in \R$. 
Keeping this in mind, direct computations show that, if $x-x'\in 2h\Z$ and  $x,x'\in \overset{\circ}{\Mc^{\kappa}_h}$, and $\bar a \in \Ab$, then
\begin{align*}  
&\frac{1}{\Delta t_{h}}\E\left[|x+\Delta {\rm x}(x,\bar a,\xi_{1})-x'-\Delta {\rm x}(x',\bar a,\xi_{1})|-|x-x'|\right]
\\
&=
 \left(|x-x'+2h|-|x-x'|\right)\frac{q^{+}_{h}(x)\wedge m-q^{+}_{h}(x)\wedge q^{+}_{h}(x')}{\Delta t_{h}} \\
 &\;\;+ \left(|x-x'-2h|-|x-x'|\right)\frac{q^{+}_{h}(x')\wedge m-q^{+}_{h}(x')\wedge q^{+}_{h}(x)}{\Delta t_{h}}\\
&=
 \left(|x-x'+2h|-|x-x'|\right)\frac{\mu(x)-\mu(x)\wedge\mu(x')}{2h} + \left(|x-x'-2h|-|x-x'|\right)\frac{\mu(x')-\mu(x)\wedge\mu(x')}{2h}\\
 &=\left[\1_{\{x\ge x'\}} (\mu(x)-\mu(x'))+ \1_{\{x'> x\}} (\mu(x')-\mu(x))\right]\\
 &\le -c|x-x'|\,.
\end{align*}
On the other hand, if $x=z_{1}$,  $x'\in \overset{\circ}{\Mc^{\kappa}_h}$ and $z_{1}-x'\in 2h\Z$, then 
\begin{align*}  
&\frac{1}{\Delta t_{h}}\E\left[|x+\Delta {\rm x}(x,\bar a,\xi_{1})-x'-\Delta {\rm x}(x',\bar a,\xi_{1})|-|x-x'|\right]\\
&= \frac{1}{\Delta t_{h}}\left(-2h q^{+}_{h}(x',\bar a)-4h(q^{+}_{h}(z_{3},\bar a))-q^{+}_{h}(x',\bar a))-2h(1-q^{+}_{h}(z_{3},\bar a))\right)\1_{\{x'\ge z_{1}+4h\}}\\
&\;\;\; -\frac{1}{\Delta t_h}\abs{x-x'}\1_{\{x'< z_{1}+4h\}}\\
&\le -c|z_{3}-x'|\\
&\le -\frac{c}{2}|x-x'|\,.
\end{align*}
In the case, $x'=z_{2\kappa+1}$ (with $\kappa\ge 4$ which we can assume here w.l.o.g.), then 
\begin{align*}  
&\frac{1}{\Delta t_{h}}\E\left[|x+\Delta {\rm x}(x,\bar a,\xi_{1})-x'-\Delta {\rm x}(x',\bar a,\xi_{1})|-|x-x'|\right]
\\
&= \frac{1}{\Delta t_{h}}\left[-4h q^{+}_{h}(x',\bar a)-6h(q^{+}_{h}(z_{3},\bar a))-q^{+}_{h}(z_{2\kappa-1},\bar a))-4h(1-q^{+}_{h}(z_{3},\bar a)))\right]\\
&\le -c|z_{3}-z_{2\kappa-1}|\\
&\le -\frac{c}{2}|x-x'|,
\end{align*}
in which the last inequalities follows from the fact that $\kappa\ge 4$.
A similar analysis can be done when $x'=z_{2\kappa+1}$ and $x\in {\Mc^{\kappa}_h}$. The above implies that, for $h$ small enough, 
\begin{align*}  
\frac{1}{\Delta t_{h}}\E\left[|x+\Delta {\rm x}(x,\bar a,\xi_{1})-x'-\Delta {\rm x}(x',\bar a,\xi_{1})|-|x-x'|\right]
\le -\frac{c}{2}|x-x'|\,\forall\; x,x'\in {\Mc^{\kappa}_h} \mbox{ s.t. } x-x'\in  2h\Z,
\end{align*}
 which is the required contraction property, whenever $x-x'\in  2h\Z$. The key property is  that $\check X^{x,\check \alpha}-\check X^{x',\check \alpha}$ remains in  $2h\mathbb{Z}$ whenever $x-x'\in 2h\mathbb{Z}$ (by the above calculations jumps of $\check X^{x,\check \alpha}-\check X^{x',\check \alpha}$ lie in $\{-6h,-4h,-2h,0,2h,4h,6h\}$). Then, the same arguments as in the proof of Lemma \ref{lemma: Lipschitz value discounted jump}   imply that one can find $\check L>0$, that only depends on $c$, such that 
\begin{align}\label{eq: check V lambda lip sur 2h} 
|\check V_{\lambda}(x)-\check V_{\lambda}(x')|\le \check L|x-x'|,\; \mbox{ for } x,x'\in {\Mc^{\kappa}_h}\mbox{ s.t. }  x-x'\in  (2h\mathbb{Z}).
\end{align}
In particular, 
\begin{align}\label{eq: check V lambda lin growth sur 2h} 
|\check V_{\lambda}(x)-\check V_{\lambda}(0)|\le \check L|x|,\; \mbox{ for } x\in {\Mc^{\kappa}_h}\cap  (2h\mathbb{Z}).
\end{align}  

2.a.(ii) We now turn to the general case in which the distance between the initial data does not belong to $2h\Z$. Take $x\in \{x_{\circ}-h,x_{\circ}+h\}\cap \overset{\circ}{\Mc^{\kappa}_h}$, for some $x_{\circ}\in {\Mc^{\kappa}_h}\cap  (2h\mathbb{Z})$.  Let $\theta_1$ be the first time at which $|\check X^{x,\check \alpha}_{\theta_{1}}-x|=h$. By the dynamic programming principle,
\begin{align*}
|\check V_{\lambda}(x)-\check V_{\lambda}(x_{\circ})|
&\le \sup_{\check\alpha\in\check{\Ac\;\,}\!\!\!}\E\left[\frac1\lambda (1-e^{-\lambda \theta_{1}})\|r\|_{\Cc^{0}_{b}}+e^{-\lambda \theta_{1}}|\check V_{\lambda}(\check X^{x,\check \alpha}_{\theta_{1}})-\check V_{\lambda}(x_{\circ})|\right]
\\&+\E\left[(1-e^{-\lambda\theta_{1}})|\check V_{\lambda}(x_{\circ})|\right]
 \end{align*}
 in which $\check X^{x,\check \alpha}_{\theta_{1}}-x_{\circ}\in \{-2h,0,2h\}$ and therefore  $|\check V_{\lambda}(\check X^{x,\check \alpha}_{\theta_{1}})-\check V_{\lambda}(x_{\circ})|\le 2\check L|h|$ by \eqref{eq: check V lambda lip sur 2h}. By exhaustive enumeration, one can compute  
 \begin{align*} 
\E[ e^{-\lambda \theta_{1}}]
&=\sum_{k\ge 1} q_{h}(x)^{k-1}(1-q_{h}(x) \left( \int_{0}^{\infty} e^{-\lambda y}\frac1{\Delta t_{h}}e^{-\Delta t_{h}^{-1} y} dy\right)^{k}\\
&=\sum_{k\ge 1} q_{h}(x)^{k-1}(1-q_{h}(x)) \left(\lambda \Delta t_{h}+1\right)^{-k}\\
&= \left(\lambda \Delta t_{h}+1\right)^{-1}(1-q_{h}(x))\frac{\lambda \Delta t_{h}+1}{\lambda\Delta t_{h}+1-q_{h}(x_\circ)}\\
&= \frac{1-q_{h}(x)}{\lambda\Delta t_{h}+1-q_{h}(x)}\le 1.
 \end{align*}
 Since $1-q_{h}(x)\ge (\sigma/L_{b_{1},b_{2}})^{2}\ge (\varsigma/L_{b_{1},b_{2}})^{2}>0$ for all $h$, by Assumption \ref{asmp: diffusion limit existence},     the above   implies that, for some $C>0$, independent on $\lambda$, $\kappa$ and $h$,
 \begin{align*}
|\check V_{\lambda}(x)-\check V_{\lambda}(x_{\circ})|
\le& \sup_{a\in\Ab}\E\left[\frac{\Delta t_{h}}{\lambda\Delta t_{h}+1-q_{h}(x)}\|r\|_{\Cc^{0}_{b}}+2\frac{1-q_{h}(x)}{\lambda\Delta t_{h}+1-q_{h}(x)}\check L|h|\right]
\\
&+\E\left[\frac{\lambda \Delta t_{h}}{\lambda\Delta t_{h}+1-q_{h}(x)}|\check V_{\lambda}(x_\circ)|\right]\\
=&C(\Delta t_{h}+h+\lambda \Delta t_{h}|\check V_{\lambda}(x_{\circ})| ). 
 \end{align*}
 Note that $\lambda \check V_{\lambda}$ is bounded by $\|r\|_{\Cc^{0}_{b}}<\infty$, while $\Delta t_{h}\le h\le |x|$, for $x\ne 0$ and $h$ small enough. 
 Since $x_{\circ}\in {\Mc^{\kappa}_h}\cap  (2h\mathbb{Z})$, the above, combined with \eqref{eq: check V lambda lin growth sur 2h},
 thus shows that 
 \begin{align}\label{eq: check V lambda lin growth sur 2h et h - bords}  
 |\check V_{\lambda}(x)-\check V_{\lambda}(0)|\le \check L'|x| ,\; \forall\;x\in \overset{\circ}{\Mc^{\kappa}_h},
 \end{align}
 for some $ \check L'>0$ that does not depend on $\lambda$, $h$ nor $\kappa$. In the case where $x \in \{z_{1},z_{2\kappa+1}\}$, we can conduct a similar analysis by considering the first time $\theta_{1}$ at which  $\check X^{x,\check \alpha}$ jumps. In this case, $\check X^{x,\check \alpha}_{\theta_{1}}\in \overset{\circ}{\Mc^{\kappa}_h}$ by construction and $|\check X^{x,\check \alpha}_{\theta_{1}}-x_{\circ}|\le 2h$. Given \eqref{eq: check V lambda lip sur 2h}, we retrieve a similar estimate as \eqref{eq: check V lambda lin growth sur 2h et h - bords}. Hence, 
 \begin{align}\label{eq: check V lambda lin growth sur 2h et h} 
|\check V_{\lambda}(x)-\check V_{\lambda}(0)|\le \check L'|x| ,\; \forall\;x\in \Mc^{\kappa}_h,
\end{align}
for some $ \check L'>0$ that does not depend on $\lambda$, $h$ nor $\kappa$.

2.a.(iii) We are now in position to show that $[\bar \Wf^{\kappa}_{h}]_{\Cc^{0}_{\rm lin}(\Mc^{\kappa}_h)}$ does not depend on $\kappa$ nor $h$.
 Using \eqref{eq: check V lambda lin growth sur 2h et h} and the arguments of Lemma \ref{lem: jump ergodic rho HJB W}, we obtain that, after possibly passing to a subsequence,
 $(\check V_{\lambda}-\check V_{\lambda}(0))_{\lambda>0}$ converges pointwise, as $\lambda \to 0$, to $\bar \Wf^{\kappa}_{h}-\bar\rho^{\kappa,*}_h\Delta t_{h}$ and that the latter satisfies 
\begin{align}\label{eq: borne bar Wc lipschitz} 
|\bar \Wf^{\kappa}_{h}(x)-\bar\rho^{\kappa,*}_h\Delta t_{h}|\le  \check L'|x|,\;x\in \Mc^{\kappa}_h.
\end{align}
 
2.b. To complete the proof, it remains to appeal to the stability of viscosity solutions,  and use comparison results in the class of semi-continuous  super/sub-solutions with linear growth. Let $(\kappa_{h})_{h>0}$ be as in the statement of the Proposition.  By \eqref{eq: borne bar Wc lipschitz}, $(\bar \Wf^{\kappa_{h}}_{h}-\bar\rho^{\kappa_h,*}_h\Delta t_{h})_{h>0}$ admits  locally bounded relaxed semi-limits 
$$
\bar \Wf^{\infty *}_{0}(x):=\limsup_{x'\to x, \;h\downarrow 0} \bar \Wf^{\kappa_{h}}_{h}(x')-\bar\rho^{\kappa_h,*}_h\Delta t_{h} \;,\;\bar \Wf^{\infty}_{0 *}(x):=\liminf_{x'\to x, \;h\downarrow 0} \bar \Wf^{\kappa_{h}}_{h}(x')-\bar\rho^{\kappa_h,*}_h\Delta t_{h} .
$$
 which take the value $0$ at $0$, recall \eqref{eq: edp diff finie matriciel normalisation rho}, and have linear growth. 
 We can then use \eqref{eq: edp diff finie matriciel}, Proposition \ref{prop : conv rho num} and standard stability arguments for viscosity solutions, see e.g.~\cite[Section 3]{fleming1989existence}, to deduce that $\bar \Wf^{\infty}_{0 *}$ and $\bar \Wf^{\infty *}_{0}$  are respectively   viscosity super- and subsolutions of \eqref{eq: HJB ergodic jump}. We claim that $\bar \Wf^{\infty *}_{0}=\bar \Wf+g$ for some $g\in \R$. Then, we will deduce that $g= \bar \Wf^{\infty *}_{0}(0)-\bar \Wf(0)=0$ by construction. The same argument can be used to prove that $\bar \Wf^{\infty}_{0*}=\bar \Wf$. To prove the above, we follow the arguments of \cite[Proof of Theorem 3.1]{barles2016unbounded}. We first fix $R>0$ and let $B_{R}:=B_{R}(0)$ be the open ball of radius $R$ centered at $0$. Set $g:=\max_{\partial B_{R}} (\bar \Wf^{\infty*}_{0}-\bar \Wf)$.
Since $\Phi:=\bar \Wf^{\infty*}_{0}-\bar \Wf-g$ has linear growth, see \eqref{eq: estimee bar W} and above, we can fix $\iota>0$, independently of $R$, such that $x\mapsto \Phi(x)-\iota |x|^{2}$ has a maximum point $\hat x_{R}$ on   $(B_{R})^{c}$. If $\sup_{(B_{R})^{c}}\Phi>0$, then, for $\iota>0$ small enough, we have $\Phi(\hat x_{R})-\iota |\hat x_{R}|^{2}>0$ and therefore $\hat x_{R}$ lies in the interior of  $(B_{R})^{c}$.
We now use the subsolution property of $\bar \Wf^{\infty *}_{0}$ and the fact that $\bar \Wf$ is a smooth solution of \eqref{eq: HJB ergodic jump} to obtain 
\begin{align*} 
0\le&  \sup_{\bar a \in \Ab}\left\{\bar \Lc^{\bar a} \bar \Wf(\hat x_{R})+r(\hat x_{R},\bar a)-\bar \rho^{*}+\iota \left(2\mu(\hat x_{R},\bar a)  \hat x_{R} + \sigma^{2}\right)\right\}
\\
\le & \iota \sup_{\bar a \in \Ab}\left\{2\mu(\hat x_{R},\bar a)  \hat x_{R} + \sigma^{2} \right\}.
\end{align*}
Using \eqref{eq: hyp R large},  we get a contradiction for $R$ large enough. This shows that $\sup_{(B_{R})^{c}}\Phi\le 0$. Now the fact that $\max_{B_{R}\cup \partial B_{R}}\Phi= 0$ follows by the maximum principle applied to \eqref{eq: HJB ergodic jump} on $B_{R}$ with Dirichlet boundary conditions on $\partial B_{R}$. Moreover, $\Phi$ is a viscosity subsolution of
\begin{align*} 
0\le&  \sup_{\bar a \in \Ab} \bar \Lc^{\bar a} \Phi.
\end{align*}
We can thus now appeal to the strong maximum principle, see e.g.~\cite[Theorem 1]{kawohl1998strong}, to deduce that $\bar \Wf^{\infty}_{0}-\bar \Wf-g=\Phi\equiv 0$, which concludes the proof.

\end{proof}
  
\section{Application to high-frequency auctions}\label{sec:numerical example}
 
 \subsection{Motivation and setting}\label{subsec: example motivation}
 
 Web display advertising is a typical example of real-world high-frequency pure jump control problems \cite{fernandez2017optimal}. The ad-spaces are sold by algorithmic platforms in automated auctions which occur at the dozen microsecond scale \cite{ostrovsky_reserve_2011}. The frequency imposes computational issues on optimisation problems in this industry, while at the same time the volume creates a significant monetary incentive for all parties to engage in revenue maximisation.

 Consequently, the question of the strategic behaviour of bidders in repeated auctions in the face of learning sellers has been a popular topic in contemporary auction theory, see e.g.~\cite[\textsection~4]{nedelec2022learning} for a survey. A rich line of work has focused on asymetric problems where one player is signifcantly more patient than the other \cite{amin2013learning, nedelec2019learning}. This asymetry reduces game theoretic considerations in the analysis to optimisation or control problems. In this example we take interest in the case where the buyer is infinitely patient (it optimises an ergodic objective), while the seller's algorithm has an effectively finite memory.
 
 Given these horizons, the format of the auction will strongly influence the behaviour of bidders and sellers when they seek to maximise their profit, see e.g.~\cite{krishna_auction_2009} for some generic examples. While it is a sub-optimal auction format for the seller \cite{myerson_optimal_1981}, we choose to focus on the second price auction format here. Indeed, there are unsurmoutable difficulties in learning the optimal auction format \cite{morgenstern_pseudo-dimension_2015}, and second price is in practice a common compromise between tractability and optimality \cite{roughgarden_minimizing_2019}. 
 
 Recalling the notations introduced in Example \ref{example : mean reverting 1}, in a second price auction (with reserve) the bidder wins if it outbids the competition $e_4$ and the reserve price $x$, and pays the smallest bid which still wins the auction, i.e.~$x\vee e_4$. 
 As a result of the time-scale there is little time in practice to perform computations to determine the bid, and one typically relies on using a precomputed a function of the value to bid when an auction arrives and $e_3$ is revealed. More formally, the bid  should be predictable. For simplicity, in this example, we consider a linear shading of the value: $a  e_2$, where the control input value $a$ is the shading factor. Consequently, we have the (expected) reward function
 \begin{align}\label{eq: auction example  reward}
    r(x,a):=\int(e_2-x\vee e_4)\1_{\{a e_2\ge x\vee e_4 \} }\nu(\de e)\,.
 \end{align}
 Such auctions are well defined only for positive bids. Thus, we impose $a\in\RR_+$.
 
 Within the constraints of a second price auction, maximising profits corresponds to tuning the reserve price $x$. Dynamically optimising the reserve price is a difficult problem even for a stationary bidder, see e.g.~\cite{croissant2020real,amin2013learning}. To simplify, we consider the mean-reverting dynamic introduced in Example \ref{example : mean reverting 1}. For some $\eta=\varepsilon^{-1}$ fixed, this dynamic is given by \eqref{eq: def X jump} with 
 \begin{align}\label{eq: jump  params of auction example}
   b:=b_\epsilon = \eps b_1 +\sqrt{\eps}b_2\mbox{ where } b_1(x,a,e) := e_1(ae_2-x)\;\mbox{ and } b_2(x,a,e):= e_1e_3.
 \end{align} 
 
In the above framework, the noise $e_1$ models seller aggressivity as an exogenous randomness, while $e_3$ models the seller's internal randomisation aimed at increasing robustness to strategic play. Under the conditions outlined in Example \ref{example : mean reverting 1}, we can choose for simplicity
$$
\nu(\de e)=\prod_{i=1}^{4} f_{i}(e_{i})  \de e_{i}
$$
in which
 \[f_1\sim \texttt{Unif}(0,1)\quad \&\quad f_3\sim \Nc(0,\sigma_0^2) \]
 with $\sigma_0=\frac12$. 
 
Second price auctions without reserve leave the most revenue on the table when the buyers are highly asymmetrical, we therefore study 
\[ f_2\sim \texttt{LogNorm}(\mu_1,\sigma_1)\quad\&\quad f_4\sim \texttt{Unif}(0,1) \]
with $\mu_1=0$ and $\sigma_1=\frac12$.
Note that empirical observations \cite{ostrovsky_reserve_2011} suggest log-normals are a realistic model for values.

 Assumption \ref{asmp: basics}, and the remaining conditions in Example \ref{example : mean reverting 1} for Assumptions \ref{asmp: controllability for jumps} and \ref{asmp: mean rever} are easily seen to hold under the above choices. Therefore, this pure jump process admits, and converges to, a diffusion limit by Theorem \ref{thm: delta epsilon bound on discounted value difference}, in particular it is easily checked that the coefficients of the limit diffusion are given by 
 \begin{align}\label{eq: diff params of auction examples}
   \mu(x,a):= \frac12\left(aC-x\right) \mbox{ and } \sigma(x):= \frac{\sigma_0}{\sqrt{3}}\,,
 \end{align}
 where $C:=\exp\left(\mu_1+\frac{\sigma_1^2}{2}\right)$. It is clear from  \eqref{eq: auction example  reward} and \eqref{eq: diff params of auction examples}, that values of $a$ larger than $1$ cannot be optimal, therefore  we  fix $\Ab=[0,1]$.

 \subsection{Numerical Resolution of the HJB Equations}
 
 Using this example motivated by high-frequency auctions we illustrate in this section the benefits of the diffusion limit problem in regards to numerical computation. We use the method detailed in Section \ref{sec: numerical resolution} to solve numerically \eqref{eq: HJB ergodic diffusive thm}, with parameters $\mu$ and $\sigma$ given by \eqref{eq: diff params of auction examples}. Throughout, we will take $\kappa_h:=h^{-1/4}$, for which $h\le \left(\frac{\sigma}{2}\right)^{\frac83}$ suffices to uphold condition \eqref{eq: cond kappa h2 le demi - L} since we have $[\mu]_{\Cc^0_{\rm lin}}\le(1+e^{1/8})/2$. Note that, with $f_1,f_2,f_3$ as above, $p$ in Example \ref{example : mean reverting 1} and Example \ref{example : mean reverting 2} can be taken to be any positive real number.
 
 
 In comparison to \eqref{eq: HJB ergodic diffusive thm}, solving \eqref{eq: HJB ergodic jump thm} with coefficients given by \eqref{eq: jump  params of auction example} is  complicated by the computation of the integral term. In many situations, when $\nu$ is a non-atomic measure with known closed form, quadrature would be the preferred method for resolution, see e.g.~\cite{cont_financial_2004}. In this example, this quadrature would be $4$-dimensional, which is somewhat expensive. 

 In contrast, the relatively simple form of the combination of independent noise sources makes Monte Carlo simulation competitive in this specific example. Fixing a grid $\Mc_{\ve,h_\ve}^{\kappa_\ve}$ analogous to the one in Section \ref{sec: numerical resolution}, we compute the empirical transition distribution $p_{N_\ve}^{\ve,h_\ve}:(x,a)\in\Mc_{\ve,h_\ve}^{\kappa_\ve}\times\Ab\to p^{\ve,h_\ve}_{N_\ve}(\cdot;x,a)\in\varDelta_{2\kappa_\ve+1}$, where $\varDelta_{2\kappa_\ve+1}$ is the $2\kappa_\ve+1$-dimensional probability simplex, based on $N_\ve$ independent samples from each law, by projecting sample transitions onto $\Mc_{\ve,h_\ve}^{\kappa_\ve}$. We then approximate for \eqref{eq: HJB ergodic jump thm} by solving the analogue of \eqref{eq: edp diff finie matriciel}, i.e. finding ($\rho_{\varepsilon,h_\ve}^{\kappa_\ve,*}$, $\Wc^{\kappa_\ve}_{\ve,h_\ve}$),  $\Wc^{\kappa_\ve}_{\ve,h_\ve}:=(\Wf_{\ve,h_\ve}^{\kappa_\ve}(z_i))_{1\le i\le 2\kappa_\ve+1}$, solving
 \begin{align}
    0&= \max_{{\rm a}\in{\rm  A}}\left\{ \frac{1}{\varepsilon}(P_{N_\ve,\ve,h_\ve}^{\rm a}-\bm{I}_{2\kappa_\ve+1})\Wc^{\kappa_\ve}_{\ve,h_\ve} - e\Wf^{\kappa_\ve}_{\ve,h_\ve}(0) +{\mathscr R}({\rm a})\right\} \label{eq: num resolution jump HJB}\\
    \rho_{\varepsilon,h_\ve}^{\kappa_\ve, *}&= \Wf^{\kappa_\ve}_{\ve,h_\ve}(0) \label{eq: num resolution jump rho}
 \end{align}
 by policy iteration, where $P_{N_\ve,\ve,h_\ve}^{{\rm a}}= (p_{N_\ve}^{\ve,h_\ve}(z_j;z_i,{\rm a}(z_i)))_{1\le i,j\le 2\kappa_\ve+1}$, $\bm{I}_{2\kappa_\ve+1}$ is the $2\kappa_\ve+1$-dimensional identity matrix, and ${\mathscr R}(\rm a)$ is as in Section \ref{sec: numerical resolution}. 
 
As $\varepsilon\to0$, all the transitions concentrate into a ball of size $\varepsilon^{\frac12}$ with a drift of size $\varepsilon$, meaning that the mesh must refine faster than $\varepsilon$, in order to avoid degeneracy. Therefore, we consider the sequence of grids $\{\Mc^{\kappa_\ve}_{\ve,h_\ve}\}_{\ve\ge0}$, with $\Mc^{\kappa_\ve}_{\ve,h_\ve}=\{y_i= -10 + (i-1)h_\ve, 1\le i \le 2\kappa_\ve+1 \}$ with $h_\ve=\ve^{\frac32}$ and $\kappa_\ve=N_\ve=20\ve^{-\frac32}$.  Note that the refinement of the grid $\Mc^{\kappa_\ve}_{\ve,h_\ve}$ as $\ve\to0$ does not imply that the accuracy of the scheme increases as $\varepsilon\to0$, the increasingly fine resolution is a \emph{cost} incurred due to $\eta_\ve$. The increase in this cost becomes impossible to maintain as $\varepsilon$ becomes small, this is illustrated by Figure \ref{fig: comp time}: it rises at a rate $\ve^{-\frac32}$.

 \begin{figure}[ht]
   \centering
   \includegraphics[width=.5\textwidth]{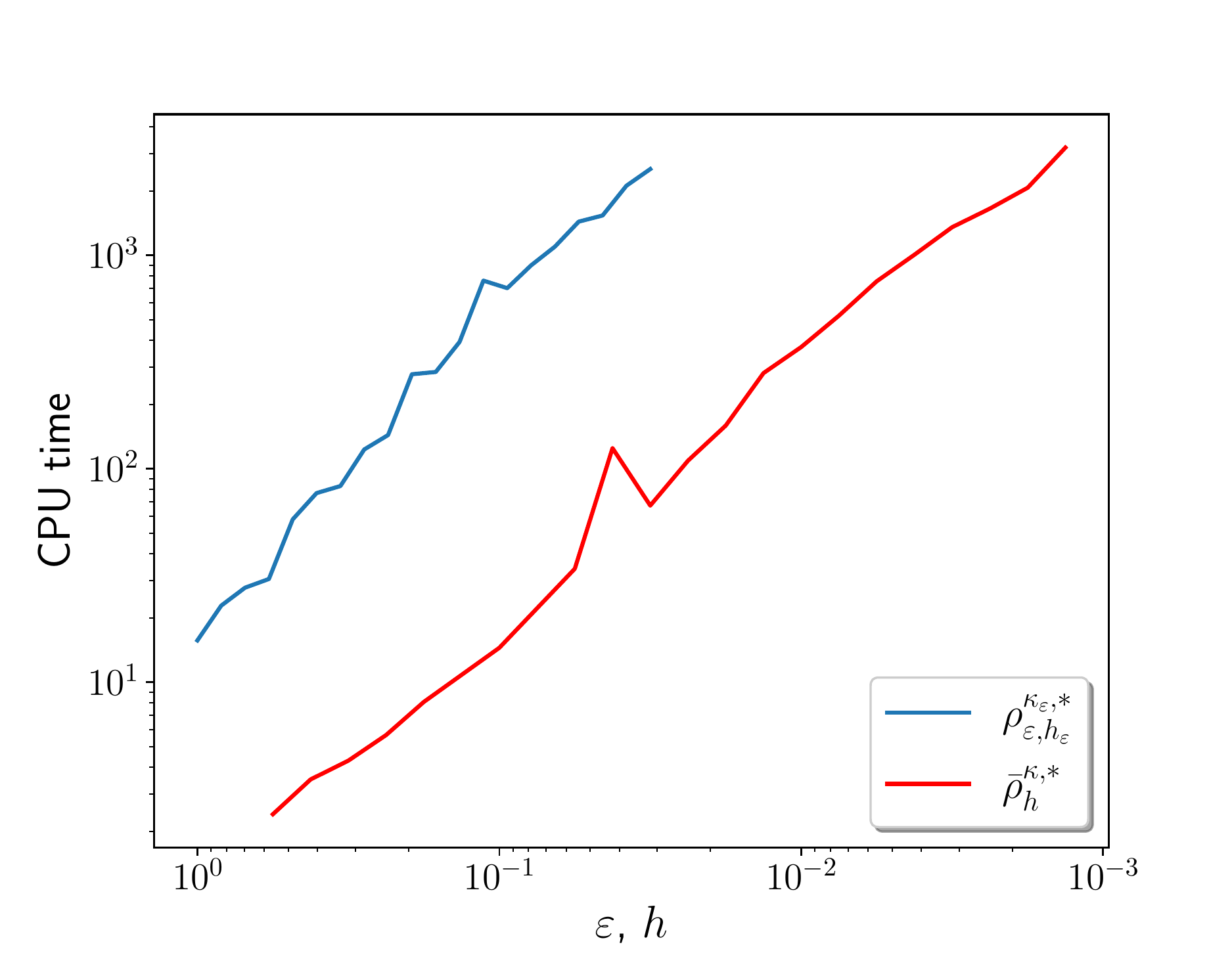}
   \caption{Comparison of computation costs for \eqref{eq: num resolution jump HJB} ($\rho_{\ve,h_\ve}^{\kappa_\ve,*}$) and \eqref{eq: edp diff finie matriciel} ($\bar\rho^{\kappa,*}_h$).}\label{fig: comp time}
 \end{figure}

In contrast, using the diffusion limit by combining Sections \ref{sec: diffusion} and \ref{sec: numerical resolution}, allows us to solve the problem to a high precision for relatively cheap. Figure \ref{fig: rho approximation} demonstrates the convergence in value of Theorem \ref{thm: delta epsilon bound on discounted value difference}, with a rate of $\ve^{\frac{1}{2}}$.
 
Explicit computation for an approximately optimal control using \eqref{eq: smoothed numerical max on a} is impractical for the $r$ given in \eqref{eq: auction example reward}, due to its lack of a closed form derivative to apply first order conditions. Nevertheless, in order to illustrate the bounds in Propositions \ref{prop: conv politique num}, we resort to numerical approximation. 
We fix a grid $\Ab_\Gamma:=\{i\Gamma^{-1},0\le i\le \Gamma\}$ on $\Ab=[0,1]$, fixing $\Gamma=100$, and then solve the maximum in \eqref{eq: smoothed numerical max on a} on $\Ab_\Gamma$ instead of $\Ab$. Contrary to Section \ref{sec: numerical resolution}, we only compute it on $\Mc^{\kappa_h}_{h}$. This yields a map $\bm{a}_h^\Gamma:\Mc_h^{\kappa_h}\to\Ab_\Gamma$, which can be viewed as a vector of controls associated to $\Mc^{\kappa_{h}}_h$.


From here, we define $\breve {\rm a}_h^\Gamma\in\mathrm{A}$ by $\breve{\rm a}_h^\Gamma:=\bm{a}_h^\Gamma(\Pi_{\Mc^{\kappa_h}_h}(\cdot))$ where $\Pi_{\Mc^{\kappa_h}_h}:\RR\to\Mc^{\kappa_h}_h$ is the projector onto the grid $\Mc^{\kappa_h}_h$, consider the solution $\breve{X}^{\kappa_h, h, \Gamma}$ of
\[\breve{X}^{\kappa_h, h, \Gamma}_\cdot= \int_0^\cdot \int_{\R^{d'}} b_\ve(\breve X^{\kappa_h,h, \Gamma}_{s-},\breve{\rm a}^{\Gamma}_{h}(\breve X^{\kappa_h,h,\Gamma}_{s-}),e)N(\de e,\de s)\,, \]
and evaluate $\rho_\ve(0,\breve \alpha^{\Gamma}_{h})$ for each $\ve$, where $\breve \alpha^{\Gamma}_{h}:=\breve{\rm a}^{\Gamma}_{h}(\breve X^{\kappa_h,h,\Gamma}_{\cdot-})$. In pratice, we fix $T=1000$  and compute 
\[ \frac\ve T\EE\left[\int_0^T r(\breve{X}^{\kappa_{h}, h, \Gamma}_{t-}, \breve{\rm a}^{\Gamma}_{h}(\breve X^{\kappa_h,h,\Gamma}_{t-}))\de N_t\right]\]
by Monte Carlo with $1000$ trajectories\footnote{Computing an ergodic average over each trajectory is very numerically expensive for small values of $\ve$, reducing the feasible amount of samples. In spite of the noise, the slope $\frac12$ is still visible in Figure \ref{fig: policy suboptimality}.} of $\breve{X}^{\kappa_{h}, h, \Gamma}$.

In spite of the noise and the simple approximate control scheme, we recover the bounds of Propositions \ref{prop: conv politique num}, in terms of $\ve$ in Figure \ref{fig: policy suboptimality}, with $h=0.002667$, the smallest $h$ on Figure \ref{fig: comp time}. Note that this convergence rate matches the one of Figure \ref{fig: rho approximation}.
 
 
 \begin{figure}[ht]
   \centering
   \begin{minipage}[t]{.45\textwidth}
       \centering
       \includegraphics[width=\textwidth]{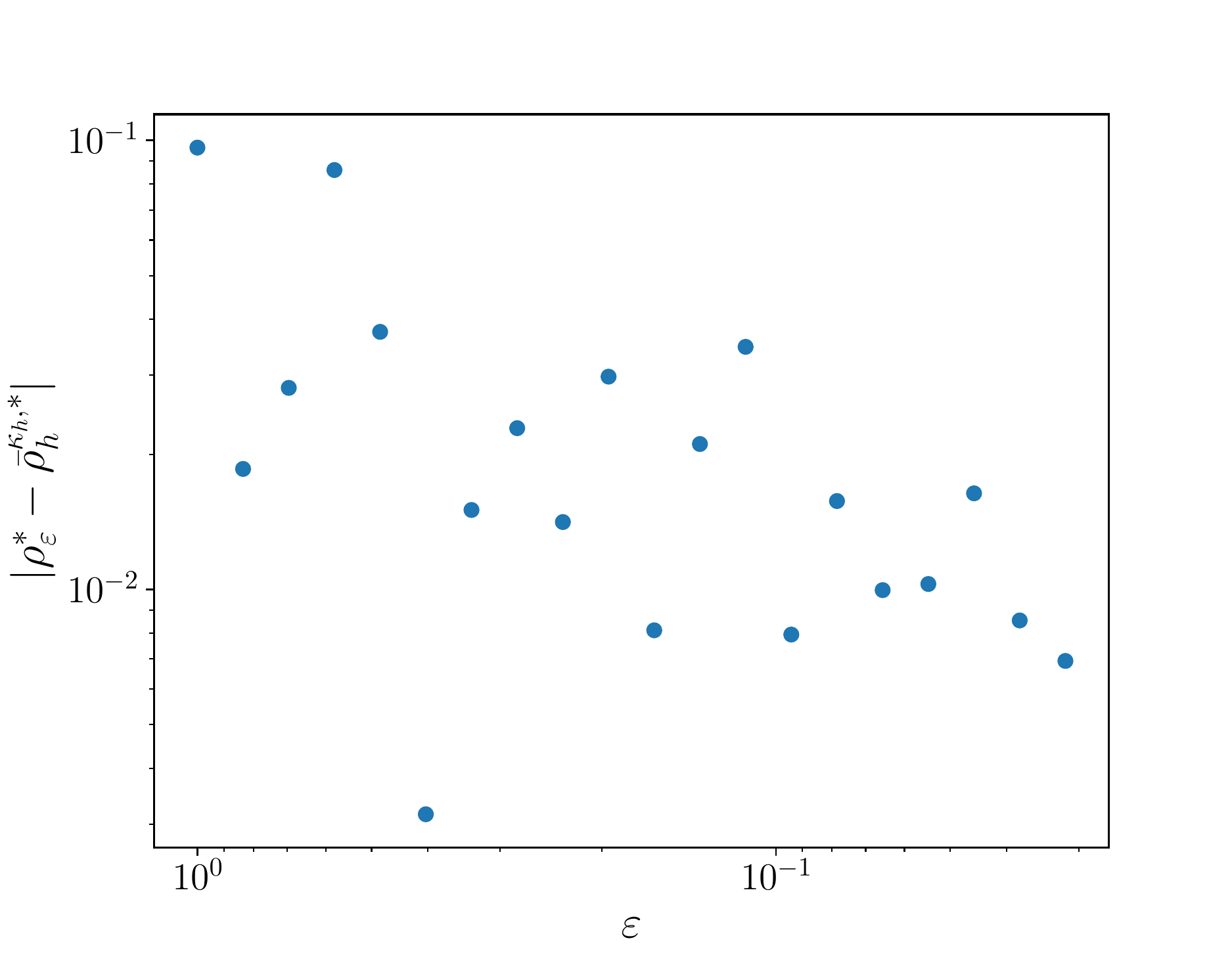}
       \caption{Approximation error of $\rho_\varepsilon^*$ by $\bar\rho_h^{\kappa_h,*}$.}\label{fig: rho approximation}
   \end{minipage}
   ~~
   \begin{minipage}[t]{.45\textwidth}
       \centering
       \includegraphics[width=\textwidth]{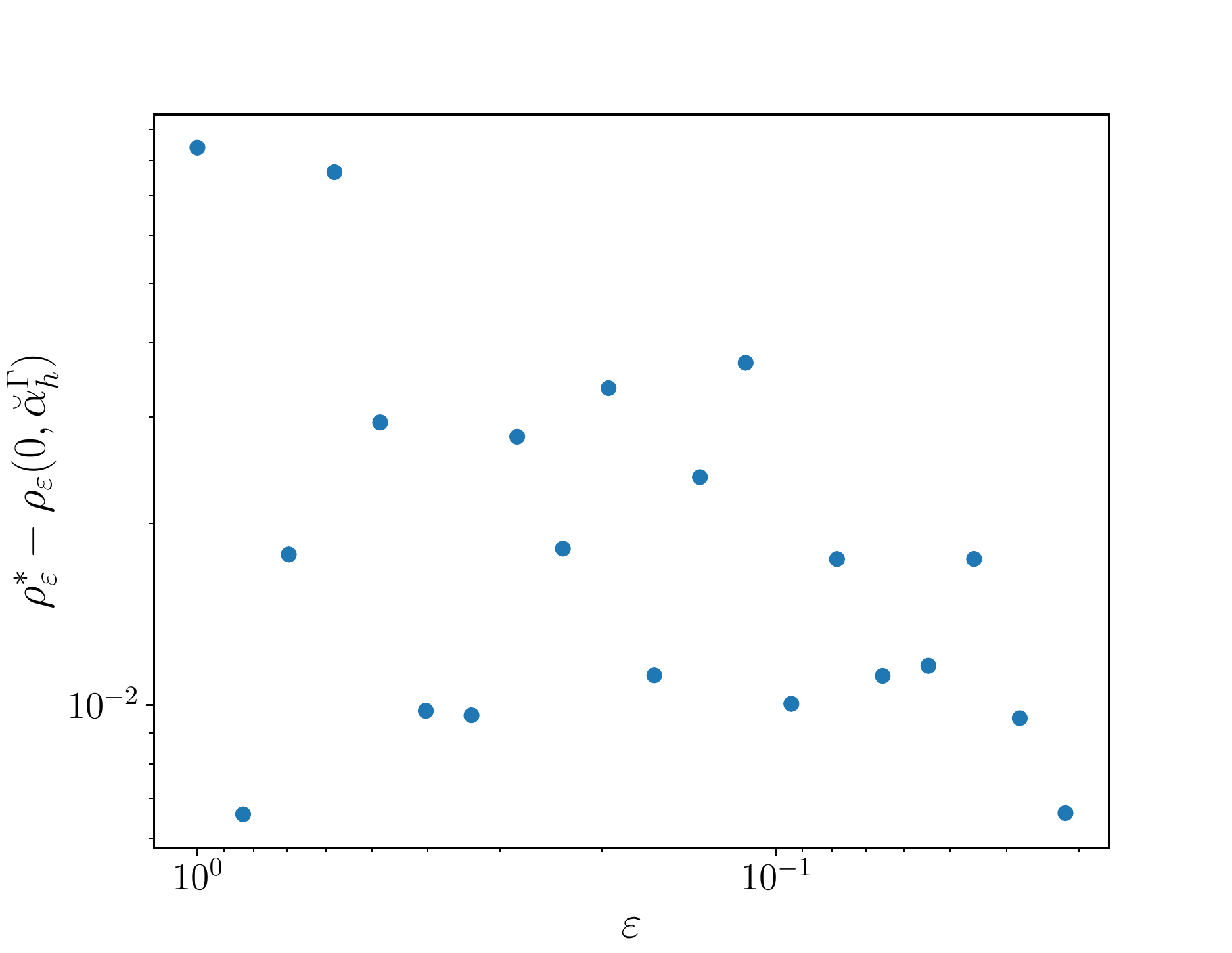}
       \caption{Suboptimality of the diffusive control relative to $\rho_\varepsilon^*$.}\label{fig: policies}\label{fig: policy suboptimality}
   \end{minipage}
 \end{figure}

\appendix
 
\section*{Appendix}
\section{Proof of Theorem \ref{thm: jump ergodic rho HJB + verif}}\label{app:jump}

In this Appendix, we first provide the proof of Theorem \ref{thm: jump ergodic rho HJB + verif}. It follows a standard route. We adapt arguments of  \cite{arisawa1998ergodic} and \cite{arapostathis2012ergodic} to our context.
 
We first show that {$(V_{\lambda})_{\lambda\in (0,1)}$ is equi-Lipschitz} continuous, under the contraction condition of Assumption \ref{asmp: controllability for jumps}.

\begin{Lemma}\label{lemma: Lipschitz value discounted jump}
Let Assumptions \ref{asmp: basics} and \ref{asmp: controllability for jumps} hold, then
$$
|V_{\lambda}(x)-V_{\lambda}(x')|\le L_{V}|x-x'|, \mbox{ for } x,x'\in \R^{d},  \lambda \in (0,1), 
$$
in which
$$
L_{V}:=\frac{  L_{{b,r}} p_\zeta}{C_{\zeta}}\left(\frac{L_\zeta}{\ell_\zeta}\right)^{\frac{1}{p_\zeta}}.
$$
\end{Lemma}

\begin{proof} Fix $x,x'\in \R^{d}$, together with $\alpha\in \Ac$. By Assumption \ref{asmp: controllability for jumps}, 
  \begin{align*} 
 & \eta\int_{\R^{d'}} \left\{\zeta (X_{{\cdot}-}+b(X_{{\cdot}-},\alpha,e),X_{-}'+b(X_{{\cdot}-}',\alpha,e))-\zeta(X_{{\cdot}-},X'_{{\cdot}-})\right\}\nu(\de e) \le -C_{\zeta}\zeta(X_{{\cdot-}},X'_{{\cdot-}})
   \end{align*}
  in which $(X,X'):=(X^{x,\alpha},X^{x',\alpha})$ and $(X_{{\cdot}-},X_{{\cdot}-}')$ is its left-limit. 
Applying It\^{o}'s Lemma then implies that 
  \begin{align*}
      \zeta(X_t,X'_t)    \le \zeta(x,x') -C_{\zeta}\int_0^{t}\zeta(X_s,X'_s)\de s +M_{t},\;t\ge 0,
  \end{align*}
  where $M$ is a local martingale. Upon using a localisation argument, {recall (i) of Assumption \ref{asmp: controllability for jumps}},  taking the expectation and using an immediate comparison result for ODEs leads to 
 \begin{align}\label{eq: pre borne contractance X-X'} 
       \EE[\zeta(X_t,X'_t)]\le \zeta(x,x')e^{-C_{\zeta} t} ,\;t\ge 0.
  \end{align}
  It remains to use (i) of Assumption \ref{asmp: controllability for jumps}  to deduce that 
  \begin{align}\label{eq: borne contractance X-X'} 
  \EE[|X_t-X'_t|^{p_{\zeta}}]\le \frac{L_{\zeta}}{\ell_{\zeta}}|x-x'|^{p_{\zeta}}e^{-C_{\zeta}t}\,,t\ge 0.
  \end{align}

Combining the above  with {Remark \ref{rem:reec V lambda et T}}, the Lipschitz continuity assumption on $r$, {Assumption \ref{asmp: basics}}, and using Jensen's inequality then leads to  
    \begin{align*}
     \abs{J_\lambda(x,\alpha)- J_\lambda(x',\alpha)}&\le    L_{b,r}\int_0^\infty e^{-\lambda t} \EE\left[ \abs{X_t-X_t'}\right] \de t\\
    &\le   L_{b,r}\int_0^\infty e^{-\lambda t}\EE\left[ |X_t-X_t'|^{p_{\zeta}}\right]^{\frac{1}{p_\zeta}}\de t\\
    & \le  L_{b,r}\left(\frac{L_\zeta}{\ell_\zeta}\right)^{\frac{1}{p_\zeta}} \int_0^\infty \abs{x-x'}e^{-\lambda t - \frac{C_{\zeta}}{p_\zeta} t}\de t\\
    &\le  \frac{L_{b,r} p_\zeta}{C_{\zeta}+\lambda p_\zeta}\left(\frac{L_\zeta}{\ell_\zeta}\right)^{\frac{1}{p_\zeta}}\abs{x-x'}\,.
    \end{align*}
    Since $\abs{V_\lambda(x) - V_\lambda(x')}\le \sup_{\alpha\in \Ac}  \abs{J_\lambda(x,\alpha)- J_\lambda(x',\alpha)}$ and $\lambda p_{\zeta}\ge 0$, this  completes the proof.
    \end{proof}
     
    We now use Assumption  \ref{asmp: mean rever} to provide a uniform (in time and the control) estimate on the diffusion \eqref{eq: def X jump}. 
     \begin{Lemma}\label{lemma: borne EXq} Let Assumptions \ref{asmp: basics} and   \ref{asmp: mean rever}  hold. Then, for all $(x,\alpha)\in \R^{d}\x \Ac$, 
\begin{align*} 
 \E[|X^{x,\alpha}_{t}|^{p_{\xi}}]&\le \frac{1}{\ell_{\xi}} \left\{e^{-C^{1}_{\xi} t}L_{\xi}|x|^{p_{\xi}}+\frac{C^{2}_{\xi}}{C^{1}_{\xi}}(1-e^{-C^{1}_{\xi} t})\right\} ,\;t\ge 0.
\end{align*}
 \end{Lemma}
 \begin{proof} Fix  $(x,\alpha)\in \R^{d}\x \Ac$ and let us write $X$ for $X^{x,\alpha}$. By \eqref{eq: Borkar contraction jump pour mean rever} and the same arguments as in the proof of Lemma \ref{lemma: Lipschitz value discounted jump}, 
 $$
 \E[\xi(X_{t})]\le  \xi(x)+ \int_{0}^{t}\E[-C^{1}_{\xi}\xi(X_{s}) +C^{2}_{\xi}]ds, \;t\ge 0,
 $$
 which implies that 
\[ \E[\xi(X_{t})]\le  e^{-C^{1}_{\xi} t}\xi(x)+\frac{C^{2}_{\xi}}{C^{1}_{\xi}}(1-e^{-C^{1}_{\xi} t}), \;t\ge 0.\]
We conclude with (i) of Assumption \ref{asmp: mean rever}. 
 \end{proof}
 
 We can now prove a first convergence result. 
\begin{Lemma}\label{lem: jump ergodic rho HJB W}
  Let Assumptions \ref{asmp: basics} and \ref{asmp: controllability for jumps} hold. Then there is $c\in \R$ and a sequence $(\lambda_n)_{n\ge1}$ going to $0$ such that $(\lambda_n V_{\lambda_n})_{n\ge 1}$ converges uniformly on compact sets to $c$, and such that   $(V_{\lambda_n} - V_{\lambda_n}(0))_{n\ge1}$ converges uniformly on compact sets  to a  function $\Wf\in \Cc^{0,1}$ that  solves
  \begin{align*}
  c&=\sup_{a\in\Ab}\left\{  \eta\int_{\R^{d'}}\left[  \Wf(\cdot+b(\cdot,a,e))-  \Wf\right]\nu(de) +r(\cdot,a)\right\},\;\mbox{ on } \R^{d},
  \end{align*}
and satisfies 
  \begin{align}\label{eq: borne lin W} 
  |\Wf(x)|\le L_{V}|x|,\;x\in \R^{d}.
  \end{align}
 \end{Lemma}
\begin{proof}
The proof applies classical arguments from \cite{arisawa1998ergodic} to the pure jump setting. By Lemma \ref{lemma: Lipschitz value discounted jump}, $(V_{\lambda} - V_{\lambda}(0))_{\lambda>0}$ is equicontinuous in the Lipschitz sense and, in particular, $\abs{V_{\lambda}(x) - V_{\lambda}(0)}\le L_{{V}}\abs{x}$ for  all $x\in\RR^d$ and $\lambda>0$. Hence, $(\lambda(V_\lambda-V_\lambda(0)))_{\lambda\ge0}$ converges uniformly on compact sets to $0$ as $\lambda\to 0$. Since $(\lambda V_\lambda(0))_{\lambda\ge0}$ is bounded, recall Lemma \ref{lemma: borne EXq} and Assumption \ref{asmp: basics}, there is a sequence $(\lambda_n)_{n\ge1}$ converging to $0$ such that $\lambda_n V_{\lambda_n}(0)\to c\in \R$  as $n\to \infty$. Thus, $\lambda_n V_{\lambda_n}\to c$ uniformly on compact sets. 

By  Lemma \ref{lemma: Lipschitz value discounted jump}, $(V_{\lambda} - V_{\lambda}(0))_{\lambda>0}$ is locally bounded. Then,  a diagonalisation argument allows one to extract a further subsequence (also denoted $(\lambda_n)_{n\ge 0}$) such that $V_{\lambda_n} - V_{\lambda_n}(0)\to \Wf$ on $\mathbb{Q}^{d}$ for some $\Wf:\mathbb{Q}^{d}\to\RR$. By the uniform equicontinuity of $(V_\lambda)_{\lambda\in (0,1)}$, $\Wf$ can be extended to $\R^{d}$ and  $V_{\lambda_n} - V_{\lambda_n}(0)\to \Wf$ uniformly on compact sets. Moreover, $\Wf$ is $L_{V}$-Lipschitz and $\Wf(0)=0$, which implies  \eqref{eq: borne lin W}. 

Next, it follows from standard arguments, see e.g.~\cite{bouchard2011weak}, that $V_{\lambda_{n}}$ solves for each $n\ge 1$
\begin{align}\label{eq: edp Vlambda} 
0= \sup_{a\in\Ab}\left\{\eta\int_{\R^{d'}} [V_{\lambda_n}(\cdot+b(\cdot,a,e)) - V_{\lambda_n}] \nu(\de e) +r(\cdot,a)\right\} -\lambda_{n} V_{\lambda_{n}},\;\mbox{ on } \R^{d}.
\end{align} 
Hence, 
\begin{align*}
\lambda_{n}V_{\lambda_{n}}(0)
=&-\lambda_{n} (V_{\lambda_{n}}-V_{\lambda_{n}}(0)) \\
&+\sup_{a\in\Ab}\left\{\eta\int_{\R^{d'}} [V_{\lambda_n}(\cdot+b(\cdot,a,e))-V_{\lambda_{n}}(0) - (V_{\lambda_n}-V_{\lambda_{n}}(0))]  \nu(\de e) +r(\cdot,a)\right\} ,\;\mbox{ on } \R^{d},
 \end{align*}
 and passing to the limit  (recall  Assumption \ref{asmp: basics} and that $\nu$ is a probability measure) implies that 
 \begin{align*}
c
=& \sup_{a\in\Ab}\left\{\eta\int_{\R^{d'}} [\Wf(\cdot+b(\cdot,a,e)) - \Wf]  \nu(\de e) +r(\cdot,a)\right\},\;\mbox{ on } \R^{d}.
 \end{align*}
 \end{proof}

 We now have to prove that the constant $c$ defined above equals  $\rho^{*}(0)$ and that only $(\Wf,\rho^{*}(0))$ solves \eqref{eq: HJB ergodic jump thm}, up to restricting to functions with linear growth taking the value $0$ at $0$.  
\begin{Lemma}\label{lem: jump ergodic convergence and verification}  
Let Assumptions \ref{asmp: basics},   \ref{asmp: controllability for jumps} and \ref{asmp: mean rever}  hold. Let  $(\tilde \Wf,\tilde \rho)\in \Cc^{0}_{\rm lin}\x \R$ be a solution of the ergodic equation 
  \begin{align*} 
\tilde \rho&=\sup_{a\in\Ab}\left\{\eta\int_{\R^{d'}}[\tilde \Wf(\cdot+b(\cdot,a,e))-\tilde \Wf] \nu(\de e)+r(\cdot,a)\right\}  ,\;\mbox{ on } \R^{d}.
  \end{align*}
 Then, $\rho^{*}$ is constant and equal to $\tilde \rho$. In particular, the constant $c$ of Lemma \ref{lem: jump ergodic rho HJB W} is equal to $\rho^{*}$.
\end{Lemma}

\begin{proof} Let us fix $x\in \R^{d}$. 

{\rm a.} By Lemma  \ref{lem: jump ergodic rho HJB W} {and \cite[Proposition 7.33, p.153]{BertsekasShreve.78}}, we can find a measurable map $x'\in \R^{d}\to \hat {\rm a}(x') \in \Ab$   such that 
\begin{align*} 
\tilde \rho =  \eta\int_{\R^{d'}} [\tilde \Wf(\cdot+b(\cdot,\hat {\rm a}(\cdot),e)) - \tilde \Wf] \nu(\de e) + r(\cdot,\hat {\rm a}(\cdot))  ,\;\mbox{ on } \R^{d}.
\end{align*}
 Let $\hat X$ denote the solution of \eqref{eq: def X jump} associated to $\hat \alpha:= \hat {\rm a}({\hat X_{\cdot -}})$ and the initial condition $x$. Then, It\^{o}'s Lemma implies that 
$$
\E\left[\tilde \Wf(\hat X_{t})-\tilde \Wf(x)+ {\frac1\eta}\int_{0}^{t} r(\hat X_{s-},\hat \alpha_{s}) \de N_s\right]={\tilde \rho t }, \;t\ge 0.
$$
Moreover, since   $\tilde \Wf$ has    linear growth, there exists $C>0$ such that 
\begin{align*}
\E[|\tilde \Wf(\hat X_{t})-\tilde \Wf(x)|]\le C\E[|\hat X_{t}|+|x|].
\end{align*}
By Lemma \ref{lemma: borne EXq}, $\E[|\hat X_{t}|]/t\to 0$ as $t\to \infty$ since $p_\xi\ge 1$. Then, the above implies that 
$$
\lim_{t\to \infty} \frac{1}{\eta t} \E\left[ \int_{0}^{t} r(\hat X_{s-},\hat \alpha_{s})  \de N_s\right]=\tilde \rho.
$$
{\rm b.} Conversely,    for any $\alpha\in \Ac$,
$$
\E\left[\tilde \Wf(  X^{x,\alpha}_{t})-\tilde \Wf(x)+  {\frac1\eta}\int_{0}^{t} r(  X^{x,\alpha}_{s-},  \alpha_{s})  \de N_s\right]\le {\tilde \rho t }, \;t\ge 0.
$$
By Lemma  \ref{lemma: borne EXq} and the linear growth of $\tilde \Wf$ again, we deduce that 
$$
\limsup_{t\to \infty }\frac{1}{\eta t} \E\left[\int_{0}^{t} r(  X^{x,\alpha}_{s-},  \alpha_{s})  \de N_s\right]\le \tilde \rho.
$$

{\rm c.} Combining a.~and b.~implies that $\tilde \rho=\rho^{*}(x)$. By arbitrariness of $x\in \R^{d}$, $\rho^{*}$ is constant.
 \end{proof}

  We are now in position to prove our second convergence result, and therefore to complete the proof of Theorem \ref{thm: jump ergodic rho HJB + verif}.
\begin{Lemma}\label{lemma: equivalence of ergodic limits} Let Assumptions \ref{asmp: basics},   \ref{asmp: controllability for jumps} and \ref{asmp: mean rever}  hold.  Then, there exists a sequence   $(T_{n})_{n\ge 1}$ going to $+\infty$ such that  $(T_{n}^{-1}V_{T_n}(0,\cdot))_{n\ge 1}$ converges  uniformly on compact sets to   $\rho^{*}(0)$.
  \end{Lemma}
  
  \begin{proof}  The proof follows from the same arguments as in \citethm[Prop. VI.1]{arisawa1998ergodic} except that in their case the convergence holds uniformly on $\R^{d}$.  Let $(\lambda_{n})_{n\ge 1}$ be as in Lemma \ref{lem: jump ergodic rho HJB W} and set $T_{n}:=\delta/\lambda_{n}$ for some $\delta\in (0,1)$, {so that $\lambda_{n}\to 0$ and $T_{n}\to \infty$ as $n\to \infty$}. Fix $x\in \R^{d}$. By Lemma \ref{lemma: Lipschitz value discounted jump} and Lemma \ref{lemma: borne EXq}, we can find $C>0$ such that 
  $\E[| V_{\lambda_{n}}(X^{x,\alpha}_{t})- V_{\lambda_{n}}(x)|]\le  C(1+|x|)$ uniformly in $\alpha \in \Ac$ and for all $x\in \R^{d}$, and $t\ge 0$. Arguing as in the proof of   \citethm[Prop. VI.1]{arisawa1998ergodic}, we then deduce from the dynamic programming principle applied to $V_{\lambda_{n}}$, see e.g.~\cite{bouchard2011weak},  {Lemma \ref{lemma: Lipschitz value discounted jump}}, Lemma \ref{lemma: borne EXq} and Assumption \ref{asmp: basics} that, for some $C'>0$ {that does not depend on $n$},  
  $$
  |\rho^{*}(1-e^{-\delta})-\frac{\delta}{T_{n}}V_{T_{n}}(0,x)|\le 2|\lambda_{n} V_{\lambda_{n}}(x)-\rho^{*}| {+\lambda_{n}C'(1+|x|)}.
  $$
  It remains to divide the above by $\delta$, send $n\to \infty$ and use  Lemmas \ref{lem: jump ergodic rho HJB W} and \ref{lem: jump ergodic convergence and verification}  to obtain that 
  $$
  \rho^{*} \frac{(1-e^{-\delta})}{\delta} \le  \liminf_{n\to \infty} \frac{1}{T_{n}}V_{T_{n}}(0,x) \le \limsup_{n\to \infty} \frac{1}{T_{n}}V_{T_{n}}(0,x)\le \rho^{*} \frac{(1-e^{-\delta})}{\delta},
  $$
  {and we conclude by arbitrariness of $\delta\in (0,1)$}. The fact that the convergence is uniform on compact sets follows from the above and Lemma \ref{lem: jump ergodic rho HJB W}.
  \end{proof}

\section{Estimates for elliptic Hamilton-Jacobi-Bellman equations without control on the volatility part}\label{app:regularity}
\def\bmf{{\mathfrak b}}
\def\amf{{\mathfrak a}}
\def\fmf{{\mathfrak f}}\def\Smf{{\mathfrak S}}

  In this section, we  collect standard estimates on  elliptic Hamilton-Jacobi-Bellman equations associated to infinite horizon optimal control problems of a diffusion, in which there is no control on the volatility part. This is a specific class of quasi-linear equations whose analysis is standard. 
 Our focus here is on the growth rate of  local $\Cc^{2,1}_{b}$-estimates  in the case where the solution is already known to be Lipschitz. We follow closely the arguments of \cite{gilbarg2015elliptic} that considers compact domains and insist only on the points where the Lipschitz continuity property is used.
  
 As usual, we first consider linear equations of the form
 \begin{align}\label{eq: PDE appendix} 
 0= \scap \bmf {\D u^\top}  + \frac12\Trb{\amf \D^{2}u}-\lambda u - \fmf \mbox{ on } \R^{d}.
 \end{align}
We fix $M>0$ and a modulus of continuity $\varrho$ (i.e.~a   real valued map on $\R^{d}$ that is continuous at $0$ and such that $\varrho(0)=0$). We let $\Smf(M,\varrho)$ denote the collections of real-valued maps $u\in \Cc^{2}$ such that  {$u(0)=0$,} $|\D u|\le M$ and  that are strong solutions of \eqref{eq: PDE appendix} with coefficients satisfying:
\begin{enumerate} 
\item[(i)] $\lambda\in [0,1]$, 
\item[(ii)] $(\bmf,\fmf) : \R^{d}\to \R^{d} \x \R$ is measurable and $[\bmf]_{\Cc^0_{\rm lin}}+\norm{\fmf}_{\Cc^0_b}\le M$,
\item[(iii)]$\amf:\RR^d\to\mathbb{S}^d$ is   bounded by $M$  and   admits $\varrho$ as a modulus of continuity, 
\item[(iv)]$\inf\{\xi^{\top} \amf\;\xi :\xi \in \R^{d},|\xi|=1\}\ge 1/M$. 
\end{enumerate} 
Hereafter, we use the convention $0/0=0$.
 
 \begin{Lemma}\label{lem: estimate linear PDE} For each ${\gamma}\in (0,1)$, there exists $K_{M,\varrho}^{\gamma}>0$ such that any $u\in \Smf(M,\varrho)$  satisfies 
 $$
 \|u\|_{\Cc^{1,\gamma}_{b}(B_{2}(x))} \le K_{M,\varrho}^{\gamma}(1+\abs{x}), \;\mbox{ for all $x\in \R^{d}$.}
 $$
\end{Lemma}

\begin{proof}  
  {\rm  1.}    Given $p>1$, we first estimate $\|u\|_{W^{2,p}(B_{2}(x))}$ in which $\|\cdot\|_{W^{2,p}(B_{2}(x))}$ denotes the norm associated to the Sobolev space $W^{2,p}(B_{2}(x))$. We follow the proof of \cite[Theorem 9.11]{gilbarg2015elliptic}. Fix $x_0\in B_2(x)$. {By \cite[(9.37)]{gilbarg2015elliptic}}, for any $v\in W^{2,p}(B_{3}({x_{0}}))$ supported in some $B_R(x_0)\subset B_3(x)$, $R>0$, there is ${C_{1}}>0$,  {that  depends only $p$,} such that 
    \[\norm{\D^2 v}_{L^{p}(B_R(x_0))} \le C_{1}M \left(\sup_{{B_R(x_0)}}\abs{\amf-\amf(x_0)}\norm{\D^2v}_{L^{p}(B_R(x_0))}  +  \norm{\Tr[\amf \D^2v]}_{L^{p}(B_R(x_0))}\right)\,,\]
 {in which $\norm{\cdot}_{L^{p}(B_R(x_0))}$ denotes the usual norm of  the  $L^{p}$-space associated to the Lebesgues measure on  $B_R(x_0)$.}    
    
  {The uniform continuity} of $\amf$ implies that there exists $R>0$  {small enough, that only depends on $p$, $M$ and $\varrho$, such that 
  $\abs{\amf-\amf(x_0)}\le  (2C_{1}M)^{-1}$ on $B_{R}(x_0)$, so that the above implies that 
    \begin{align}
    \norm{\D^2 v}_{L^{p}(B_{R}(x_0))} \le  2 C_{1}M    \norm{\Tr[\amf \D^2v]}_{L^{p}(B_{R}(x_0))}.\label{eq: app fixed point D2v bound}
    \end{align}}
  Take $u\in\Smf(M,\varrho)$ a solution to \eqref{eq: PDE appendix} in $B_3(x)$, applying \eqref{eq: app fixed point D2v bound} yields
    \begin{align*}
    \norm{\D^2u}_{L^{p}({{B_{R}}}(x_0))}\le {C_{2}} (\norm{\fmf}_{\Cc^0_b(B_3(x))} + \lambda\norm{u}_{\Cc^0_b(B_3(x))} +\norm{\bmf}_{\Cc^0_b(B_3(x))} \norm{\D u^{\top}}_{\Cc^0_b(B_3(x))})
    \end{align*}
  for some ${C_{2}}>0$ {that only depends on $M$, $p$ and $\varrho$}.
  From the definition of $\Smf(M,\varrho)$, it follows that there is ${C_{3}>0}$, independent of $x_{0}$,  such that
  \[ \norm{u}_{W^{2,p}(B_{{R}}(x_0))}\le  {C_{3}}(1+\abs{x})\,, \]
  and, by covering $B_2(x)$ with finitely many balls of radius {less that $R$}, one obtains 
  \[\norm{u}_{W^{2,p}(B_{2}(x))}\le  {C_{4}}(1+\abs{x}) \]
  for some ${C_{4}}$ that depends only on $p$, $M$ and $\varrho$. 

{\rm 2.}
  Using an imbedding theorem, see \eg~\cite[Theorem 7.26]{gilbarg2015elliptic}, we can find $\bar K^{\gamma,p}>0$ such that
  $$
  \|u\|_{\Cc^{1,\gamma}_{b}(B_{2}(x))} \le \bar K^{\gamma,p} \|u\|_{W^{2,p}(B_{2}(x))}, \;\forall\;u\in \Smf(M,\varrho),\;x\in \R^{d},
  $$
  for all  $p\in \N$ such that $0<d/p<1$ and $\gamma\in (0,1-d/p)$. Given $\gamma\in (0,1)$, the required result follows by combining the above for some $p$ large enough.  \end{proof}


We now turn to the quasilinear case 
 \begin{align}\label{eq: PDE appendix quasi} 
 0=   \hat \bmf(\cdot,\D u^{\top}) + \frac12\Trb{\amf \D^{2}u}-\lambda u   \mbox{ on } \R^{d},
 \end{align}
 in which 
 $$
    \hat \bmf(x,y) := \scap{ \bmf(x,y)}{y}- \fmf(x,y),\;(x,y)\in \R^{d}\times \R^{d}. 
 $$
 We again fix $M>0$, and $\rho=(\rho_{1},\rho_{2})\in (0,1]^{2}$, and let $\tilde \Smf(M,\rho)$ denote the collection of real-valued maps  {$u\in \Cc^{2}$ such that  {$u(0)=0$,} $|\D u|\le M$}, and that are   solutions of \eqref{eq: PDE appendix quasi} for some coefficients satisfying:
 \begin{enumerate} 
\item[(a.)] $\lambda\in [0,1]$, 
\item[(b.)] $(\bmf,\fmf) : \R^{d}\to \R^{d} \x \R$ is measurable and $[\bmf]_{\Cc^0_{\rm lin}(\R^{2d})}+\norm{\fmf}_{\Cc^0_b(\R^{2d})}\le M$,
\item[(c.)] $\amf:\RR^d\to\mathbb{S}^d$ is measurable and bounded by $M$. 
 \item[(d.)] $\inf\{\xi^{\top} \amf\;\xi :\xi \in \R^{d},|\xi|=1\}\ge 1/M$,
 \item[(e.)]  for all $x,x'\in \R^{d}$ such that $|x-x'|\le 1$ and all $y,y'\in \R^{d}$:
 $$|\amf(x)-\amf(x')|+|\hat \bmf(x,y)-\hat \bmf(x',y')| \le M\left(|x-x'|^{\rho_{1}}+|y-y'|^{\rho_{2}}\right).$$
  
 \end{enumerate}
 
\begin{Lemma}\label{lem: estimate linear PDE 2} Fix    $\gamma\in (0,\rho_{1}\wedge \rho_{2}{)}$. Then, there exists   $\tilde K_{M,\rho}^{\gamma}>0$ such that  any $u\in \tilde \Smf(M,\rho)$  satisfies  
 $$
 \|u\|_{\Cc^{2,\gamma}_{b}(B_{1}(x))}  \le \tilde K_{M,\rho}^{\gamma}(1+\abs{x}), \;\mbox{ for all $x\in \R^{d}$.} 
 $$ 
\end{Lemma}

\begin{proof} Fix $x\in \R^{d}$. Since $|\D u|\le M$, by Lemma \ref{lem: estimate linear PDE} applied to the coefficient $x'\in \R^{d} \mapsto (\bmf(x',\D u(x')),\amf(x'),$ $\fmf(x',\D u(x')))$ in place of $(\bmf,\amf,\fmf)$, for each $\gamma\in (0,1)$, we can find $C_{\gamma}>0$ such that 
\begin{align}\label{eq: birne C1r u} 
\|u\|_{\Cc^{1,\gamma}_{b}(B_{2}(x))}\le C_{\gamma}(1+|x|)\; \mbox{ for all $x\in \R^{d}$.}
\end{align}
It then follows from \cite[Theorem 9.19]{gilbarg2015elliptic} that  $u\in \Cc_{b}^{2,\gamma}(B_{2}(x))$ for any $\gamma \in (0,\rho_{1}\wedge \rho_{2}{)}$. 
\\
To obtain an associated estimate, we turn to the proof of \cite[Theorem 6.2]{gilbarg2015elliptic} which we apply to the solution $w= u$ of the linear equation $ Lw:= \frac12\Trb{\amf \D^{2}w}=-\hat \bmf(\cdot,\D w^{\top}) + \lambda w $, in our particular setting. Fix $x_0\in B_2(x)$, and consider the constant coefficient equation $L_0w:= \frac12\Trb{\amf(x_0) \D^{2}w}= F$ where $F(z):=\frac12\Trb{(\amf(x_0)-\amf(z) )\D^{2}{u}(z)}-\hat\bmf(z,\D {u}^{\top}(z))+\lambda {u}({z})$, $z\in \R^{d}$.

We first introduce some notations. For $\Omega\subset\RR^d$, $\gamma\in {(0,1)}$, and $f\in\Cc^{2,\gamma}(\Omega)$ define the following norm and Schauder semi-norm respectively as follows:
\begin{align}
\abs{f}_{0,\gamma,\Omega}^{(2)}&:= \sup_{z\in\Omega}d_z^2\abs{f(z)}+ \sup_{(z,z')\in\Omega^2}d_{z,z'}^{2+\gamma}\frac{\abs{f(z)-f(z')}}{\abs{z-z'}^\gamma} \notag \\
[f]^*_{2,\gamma,\Omega}&:=\sup_{\substack{(z,z')\in\Omega^2}} d_{z,z'}^{2+\gamma} \frac{\abs{\D^{{2}} f(z)-\D^{{2}} f(z')}}{\abs{z-z'}^\gamma}\\
[f]^*_{2,\Omega}&:=\sup_{z\in\Omega}d_z^{2}\abs{\D^{2} f(z)}\,, \label{eq: def Schauder semi-norm}
\end{align}
where $d_z$ is the distance of $z$ to the boundary of $\Omega$ and $d_{z,z'}:=d_z{\wedge} d_{z'}$  for any $(z,z')\in\Omega^2$.  

We now fix  $\gamma\in {(0,\rho_{1}\wedge \rho_{2})}$. Let $\mu\in(0,\frac12]$ and  {set} $ {\Omega:=}B_2(x)$. Fix $y_0\in B_2(x)$ such that $d_{x_{0}}\le d_{y_{0}}$ (without loss of generality) and set $B:=B_{\mu d_{x_{0}}}(x_{0})$. Then, \cite[Lemma 6.1 (a.)]{gilbarg2015elliptic} ({see \cite[(6.16)]{gilbarg2015elliptic} for details}) applied to $L_0w=F$ implies that
\begin{align*}
  d_{x_0,y_{0}}^{2+\gamma}\frac{\abs{\D^2 {u}(x_0)-\D^2 {u}(y_0)}}{\abs{x_0-y_0}^\gamma}&= d_{x_0}^{2+\gamma}\frac{\abs{\D^2 {u}(x_0)-\D^2 {u}(y_0)}}{\abs{x_0-y_0}^\gamma}\\
  &\le \frac{{C^{\gamma}_{1}}}{\mu^{2+\gamma}}(\norm{u}_{\Cc^0_b(B_2(x))}+\abs{F}^{(2)}_{0,\gamma,B})+\frac{{4}}{\mu^\gamma}{[u]^*_{2,B_{2}(x)}}\,
\end{align*}
for some $C^{\gamma}_1>0$, {which only depends on } $ \gamma\in {(0,\rho_{1}\wedge \rho_{2})}$.  {Then, using \cite[(6.8)]{gilbarg2015elliptic}} yields 
\begin{align*}
d_{x_0,y_{0}}^{2+\gamma}\frac{\abs{\D^2 {u}(x_0)-\D^2 {u}(y_0)}}{\abs{x_0-y_0}^\gamma}\le& {\frac{C^{\gamma}_{1}}{\mu^{2+\gamma}}}\left(\norm{u}_{\Cc^0_b(B_2(x))} + \abs{F}^{(2)}_{0,\gamma,B}\right) \\
& +{4\left(C_{1}(\mu)\|u\|_{\Cc^{0}_{b}(B_{2}(x))}+\mu^{\gamma}[u]^*_{2,\gamma,B_{2}(x)}\right)}\,
 \end{align*}
for some $C_{1}(\mu)>0$ that only depends on $\mu$.
The Schauder estimate then comes from bounding term by term $\abs{F}^{(2)}_{0,\gamma,B}$. First, we {argue as for \cite[(6.19)]{gilbarg2015elliptic}, using  (c.) and (e.) in the definition of   $\tilde \Smf(M,\rho)$, to  obtain}
\begin{align*}
  \abs{\Trb{(\amf(x_0)-{\amf}) ){\D^{2}u}}}^{(2)}_{0,\gamma,B}&\le  {C^{\gamma}_{2}\mu^{2+\gamma}}\left[ C_{2}(\mu)\|u\|_{\Cc^{0}_{b}(B_{2}(x))}+\mu^{\gamma} [u]^*_{2,\gamma,B_2(x)}\right]
  \end{align*}
for some $C^{\gamma}_{2}, C_{2}(\mu)>0$ which only depend on $\gamma$ and $\mu$.     Second, we combine \eqref{eq: birne C1r u} with items (a.) and (e.) in the definition of   $\tilde \Smf(M,\rho)$ to obtain that 
\begin{align*}
\abs{\hat{\bmf}(\cdot,\D {u}^{\top}) {-\lambda u}}^{(2)}_{0,\gamma,B_2(x)}& \le  C^{\gamma}_3(1+\abs{x})
\end{align*}
for some $C^{\gamma}_3>0$, that only depends on $\gamma$. 

{Combining the above with \eqref{eq: birne C1r u} and using the arbitrariness of $x_{0},y_0\in B_2(x)$ leads to}
\begin{align*} 
[u]^*_{2,\gamma,B_2(x)}\le& \frac{C^{\gamma}_{1}}{\mu^{2+\gamma}}\left( \norm{u}_{\Cc^0_b(B_2(x))} +C^{\gamma}_3(1+\abs{x})
\right)+ \frac{C^{\gamma}_{1}C^{\gamma}_{2}}{2}\left[ C_{2}(\mu)\|u\|_{\Cc^{0}_{b}(B_{2}(x))}+\mu^{\gamma} [u]^*_{2,\gamma,B_2(x)}\right] \\
&+ 4\left(C_{1}(\mu)\|u\|_{\Cc^{0}_{b}(B_{2}(x))}+\mu^{\gamma}[u]^*_{2,\gamma,B_{2}(x)}\right)
 \end{align*}
We now take $\mu>0$ small enough and recall \eqref{eq: birne C1r u}  to obtain, for each  $0<\gamma<  {\rho_1\wedge \rho_2}$, a constant $C^{\gamma}_4>0$, independent on $x$, such that 
\[ 
[u]^*_{2,\gamma,B_2(x)}\le C^{\gamma}_{4}\left(1+|x|\right)\,
\]
and we conclude by using  \cite[(6.9)]{gilbarg2015elliptic} and the fact that the distance between a point of  $B_{1}(x)$ and the boundary of $B_{2}(x)$ is a least $1$.
\end{proof}

 
\def\cprime{$'$} \def\cprime{$'$}

\end{document}